\documentclass[aop,MSNbibl,seceqn,citesort,dvips]{arximspdf}
\usepackage{mathbh}

\doi{10.1214/11-AOP705}
\volume{41}
\issue{2}
\pubyear{2013}
\firstpage{817}
\lastpage{847}

\makeatletter
\newtheorem{theorem}{Theorem}[section]
\newtheorem{lemma}[theorem]{Lemma}
\newtheorem{proposition}[theorem]{Proposition}
\newtheorem{corollary}[theorem]{Corollary}
\newproclaim{condition}{Condition}
\newproclaim{fact}{Fact}
\newproclaim{remark}{Remark}

\def\dist{\operatorname{dist}}

\def\rf{\rfloor}

\def\lf{\lfloor}

\def\a{\alpha}
\def\b{\beta}
\def\g{\gamma}
\def\t{\tau}
\def\th{\theta}
\def\D{\Delta}
\def\L{\Lambda}
\def\O{\Omega}
\def\S{\Sigma}
\def\del{\partial}
\def\R{\mathbb{R}}
\def\N{\mathbb{N}}
\def\P{\mathbb{P}}
\def\Z{\mathbb{Z}}
\def\E{\mathbb{E}}

\def\AA{\mathcal{A}}
\def\BB{\mathcal{B}}
\def\DD{\mathcal{D}}
\def\EE{\mathcal{E}}
\def\FF{\mathcal{F}}
\def\GG{\mathcal{G}}
\def\LL{\mathcal{L}}
\def\MM{\mathcal{M}}
\def\OO{\mathcal{O}}
\def\PP{\mathcal{P}}
\def\QQ{\mathcal{Q}}
\def\RR{\mathcal{R}}
\def\SS{\mathcal{S}}
\def\VV{\mathcal{V}}
\def\VV{\mathcal{ V}}
\def\WW{\mathcal{W}}

\def\L{{\Lambda}}
\def\b{{\beta}}
\def\D{{\Delta}}
\def\t{{\tau}}
\def\th{{\theta}}
\def\g{{\gamma}}
\def\a{{\alpha}}
\def\O{{\Omega}}
\def\th{{\theta}}
\def\del{{\partial}}

\def\D{\mathbb{D}}

\def\cov{\operatorname{Cov}}

\def\1{\mathbh{1}}
\def\wt{\widetilde}
\def\wh{\widehat}

\def\cadlag{\mbox{c\`{a}dl\`{a}g }}

\makeatother

\begin{document}
\begin{frontmatter}

\title{Convergence of clock processes in random environments and
ageing in the \lowercase{$p$}-spin SK model}
\runtitle{Convergence of clock processes}

\begin{aug}
\author[A]{\fnms{Anton} \snm{Bovier}\thanksref{t1}\ead[label=e1]{bovier@uni-bonn.de}}
\and
\author[B]{\fnms{V\'eronique} \snm{Gayrard}\corref{}\ead[label=e2]{veronique@gayrard.net}}
\runauthor{A. Bovier and V. Gayrard}

\thankstext{t1}{Supported by the DFG through SFB 611 and the Hausdorff Center
of Mathematics.}

\affiliation{Rheinische Friedrich-Wilhelms-Universit\"{a}t and Universit\'e de Provence}
\address[A]{Institut f\"ur Angewandte Mathematik\\
Rheinische Friedrich-Wilhelms-Universit\"{a}t\\
Endenicher Allee 60\\
53115 Bonn\\
Germany\\
\printead{e1}} 
\address[B]{CMI, LAPT\\
Universit\'e de Provence\\
39, rue F. Joliot Curie\\
13453 Marseille cedex 13\\
France\\
\printead{e2}}
\end{aug}

\received{\smonth{8} \syear{2010}}
\revised{\smonth{7} \syear{2011}}

%
\begin{abstract}
We derive a general criterion for the convergence of clock
processes in random dynamics in random environments that is
applicable in cases when correlations are not negligible,
extending recent results by Gayrard [(2010), (2011), forthcoming],
based on
general criterion for convergence of sums of
dependent random variables due to Durrett and Resnick
[\textit{Ann. Probab.} \textbf{6} (1978) 829--846].
We demonstrate the power of this criterion by applying it to the case of
random hopping time dynamics of the $p$-spin SK model. We prove that
on a wide range of time scales, the clock process converges to
a stable subordinator \textit{almost surely} with respect to the
environment. We also show that a time-time correlation function
converges to the arcsine law for this subordinator, almost surely.
This improves recent results of Ben~Arous, Bovier and {\v{C}}ern{\'y}
[\textit{Comm. Math. Phys.} \textbf{282} (2008) 663--695]
that obtained similar convergence results in law, with respect to the
random environment.
\end{abstract}

%

\pdfkeywords{82C44, 60K35, 60G70, Random dynamics,
random environments, clock process, Levy processes,
spin glasses, aging}

\end{frontmatter}

\section{Introduction and main results}
\label{S1}

Over the last decades, random motion in random environments have been
one of the main foci of research in applied probability theory and mathematical
physics. This is due to the wide range of real life systems that
can be modeled in this way, but also to the exciting, unforeseen
and often counter-intuitive effects they exhibit. In fact, the early
works of Solomon~\cite{Solo} and Sinai~\cite{Sinai-rwre} on
random walks in one-dimensional random environment were already striking
examples of this feature.

While the most straightforward model class, the random walk in random
environments on the lattice $\Z^d$, received the bulk of
attention in the probability community, over the last decade, the
study of the dynamics of spin glass models has attracted
considerable attention in connection with the concept of aging.
See, for example, \cite
{BC06} for a review. The dynamics of these models is expected to
show very slow convergence to equilibrium, measurable in the anomalous
behavior of certain time-time correlation functions.

Interesting models of the dynamics of spin glasses are Glauber dynamics
on state spaces $\S_n=\{-1,1\}^n$, reversible with respect to
Gibbs measures associated to random Hamiltonians, given by
correlated Gaussian processes indexed by the hypercube $\S_n$. Even
on the
nonrigorous level, predictions on their behavior were mostly based
on the basis of drastically simplified \textit{trap models}
\cite{Bou92,BD95,MB96,RMB00,BCKM98}, based in turn on the
ideas of Goldstein~\cite{Gold69} to
describe dynamics on long times scales in terms
of thermally activated barrier crossings.

A rigorous analysis of many variants of such models was carried out over
the last years~\cite{BC05,BCM06,BC06b,BC07}. A striking feature
that emerged in these works was the universal recurrence of the
$\a$-stable L\'evy subordinators as basic random mechanisms in the description
of the asymptotic properties of their dynamics.
Another line of research tried to give a rigorous justification of the
connection between spin glass dynamics and trap models. This was
successful for the \textit{Random Energy Model (REM)} of Derrida under
a particular variant of the Glauber dynamics (the random hopping time
dynamics, see below), first on times scales close to
equilibrium~\cite{BBG02,BBG03,BBG03b} and later also on
shorter time scales~\cite{BC06b}. These results were partially
extended to spin glasses with nontrivial correlations,
the so-called $p$-spin SK models, by
Ben Arous, Bovier and \v Cern\'y~\cite{BBC08}. Their results cover a limited
range of times scales (in fact one expects a change of behavior at
longer scales), and only
in law with respect to the random environment, which in this case appears
unnatural.

The recurrent appearance of stable subordinators in such a large
variety of
model systems asks for a simple and robust explanation. Such an
explanation was given in a limited context of trap models by
Ben Arous and \v Cern\'y~\cite{BC06b}.

A more direct and general view on this problem was
presented in a recent paper by one of us~\cite{G10a} and applied to
more complicated situations in~\cite{G10b} and~\cite{G10c}.
It emerges that the entire problem links up directly to a classical
and well-studied field of probability theory, the convergence of
sums of random variables to L\'evy processes. The case of independent random
variables has been well known since the work of Gnedenko and Kolmogorov
\cite{GneKol}, but a lot of work was done for the case of dependent
random variables as well. In particular, there is a very amenable and useful
criterion due to Durrett and Resnick~\cite{DR78} that we will rely on here.

Before entering in more detail, let us briefly describe the general
setting of
\textit{Markov jump processes} in random environments that we consider
here. Our arena is a sequence of loop-free graphs,
$G_n(\VV_n, \LL_n)$ with set of vertices, $\VV_n$, and set of edges,
$\LL_n$.

A \textit{random environment} is a family of
positive random variables, $\t_n(x), x\in\VV_n$, defined on
some abstract probability space, $(\O,\FF, \P)$.
Note that
we do not assume independence.

Next we define discrete time Markov processes, $J_n$, with state space
$\VV_n$
and nonzero transition probabilities along the edges, $\LL_n$.
We denote by $\mu_n$ its initial distribution and by $p_n(x,y)$ the
elements of its transition matrix. Note that the
$p_n$ may be random variables on
the space $(\O,\FF,\P)$. We assume that the process $J_n$ is reversible
and admits a unique
invariant measure $\pi_n$.

We construct our process of interest, $X_n$, as a time change of
$J_n$. To this end we set
\begin{equation}\label{inv-ab.1}
{\lambda}_n(x)\equiv C\pi_n(x)/\t_n(x),
\end{equation}
for some (model dependent) constant $C>0$,
and define the \textit{clock process}
\begin{equation}
\wt S_n(k)=\sum_{i=0}^{k-1}{\lambda}_n^{-1}(J_n(i))e_{n,i} ,\qquad
k\in\N,
\label{1.1.6}
\end{equation}
where $(e_{n,i} ,n\in\N, i\in\N)$ is a family of independent mean
one exponential\setcounter{footnote}{1}\footnote{One can consider more general situations when
$e_{n,i}$ have different distributions as well, leaving the setting of
Markov processes.}
random variables, independent of $J_n$.

We now define our continuous time process of interest, $X_n$, as
\begin{equation}
X_n(t)= J_n(i),\qquad\mbox{if } \wt S_n(i)\leq t<\wt
S_n(i+1)
\mbox{ for some } i.
\label{1.1.7}
\end{equation}
One can readily verify that $X_n$ is a continuous time Markov process
with infinitesimal generator ${\lambda}_n$,
whose elements
are
\begin{equation}\label{clock-ab.2}
{\lambda}_n(x,y)= {\lambda}_n(x)p_n(x,y),
\end{equation}
and whose unique invariant measure is given by
\begin{equation}\label{clock-ab.3}
C\pi_n(x){\lambda}^{-1}_n(x)= \t_n(x).
\end{equation}
Note that the numbers ${\lambda}_n^{-1}(x)$ play the role of the mean holding
time of the process~$X_n$ in a site $x$.

For future reference, we refer to the ${\sigma}$-algebra generated by the
variables $J_n$ and $X_n$ as $\FF^J$ and $\FF^X$, respectively.
We write $P_{\mu_n}$ for the law of the process~$J_n$, conditional on
the ${\sigma}$-algebra $\FF$,
that is, for fixed realizations of the random environment.
Likewise we call $\PP_{\mu_n}$ the law of $X_n$ conditional on~$\FF$.

This construction brings out the crucial role played by the clock
process. If the chain $J_n$ is rather fast mixing, convergence to
equilibrium can only be slowed through an erratic behavior of the
clock process. This process, on the other hand, is a sum of positive
random
variables, albeit in general dependent ones. The approach of~\cite{G10a}
(and already~\cite{BBC08}) is to abstract from all other issues
and to focus on the analysis of the asymptotic behavior of the clock
process. From that point onward, it
is not surprising that stable subordinators will emerge\vadjust{\goodbreak} as a standard
class of limit processes; the universality appearing here is simply
linked to the universal appearance of stable processes in the theory of
sums of random variables.

In this paper we are mainly concerned with establishing
criteria for the convergence of processes like (\ref{1.1.6}) under suitable
scaling; that is, we will ask when there are constants, $a_n,c_n$, such that
the process
\begin{equation}
S_n(t)\equiv c_n^{-1}\wt S_n(\lfloor a_nt\rfloor)
=c_n^{-1}\sum_{i=0}^{\lfloor a_n t\rfloor-1}
\lambda_n^{-1}(J_n(i))e_{n,i} ,\qquad t>0,
\label{1.clock}
\end{equation}
converges in some sense to a limit process. Note that in physical terms,
the constants $c_n$ correspond to the time scale on which we observe our
continuous time Markov process $X_n$, while $a_n$ corresponds to the number
of steps the underlying process $J_n$ makes during that time.

Due to the doubly stochastic nature of our processes, convergence can be
considered in various modes, that is, under various laws. The physically
most desirable one is referred to as \textit{quenched}, that is, to say
$\P$-almost sure convergence (to a deterministic or random process)
under the
law $\PP_{\mu_n}$. In~\cite{BBC08} another point of view was taken, namely
$P_{\mu_n}$-almost sure convergence under the law of the random medium and
the exponential random variables $e_{n,i}$.
Both imply the weakest form of convergence in law under the joint law of
all random variables involved, often misleadingly referred to as
\textit{annealed}. The method used in~\cite{BBC08} was based on the
analysis of
the Laplace transform of the clock process and the use of Gaussian comparison
theorems. This left no way to deal with a fixed random environment.
We will see, however, that we are to use heavily the computations from that
paper.

\subsection{Key tools and strategy}

This approach is based on a powerful and illuminating method developed
by Durrett and Resnick~\cite{DR78}
to prove functional limit theorems for dependent variables.
We state their theorem in a specialized form suitable for our applications,
which is taken from~\cite{G10a} (see Theorem 2.1).

\begin{theorem}\label{dr.1}
Let $Z_i^n$ be a triangular array of random variables with support in
$\R_+$
defined on some probability space $(\O,\FF,\PP)$.
Let $\nu$ be a sigma-finite measure
on $(\R_+,\BB(\R_+))$, such that
$\int_0^\infty(x\wedge1)\nu(dx)<\infty$.
Assume that there exists a sequence $a_n$, such that
for all continuity points $x$ of the distribution function of
$\nu$, for all $t>0$, in $\PP$-probability,
\begin{equation}\label{dr.2}
\lim_{n\uparrow\infty}\sum_{i=1}^{ \lf a_nt\rf}\PP
(Z_i^n>x|\FF
_{n,i-1}) =t\nu(x,\infty),
\end{equation}
and
\begin{equation}\label{dr.3}
\lim_{n\uparrow\infty}\sum_{i=1}^{\lf a_nt\rf}[\PP
(Z_i^n>x|\FF
_{n,i-1})]^2 =0,\vadjust{\goodbreak}
\end{equation}
where $\FF_{n,i}$ denotes the ${\sigma}$-algebra generated by the random
variables $Z_{n,j},\break j\leq i$.
If, moreover,
\begin{equation}\label{careful.20}
\lim_{{\varepsilon}\downarrow0} \limsup_{n\uparrow\infty}
\sum_{i=1}^{\lf a_n t\rf} \EE\1_{Z_i^n\leq{\varepsilon}}Z^n_i =0,
\end{equation}
then
\begin{equation}\label{dr.5}
\sum_{i=1}^{ \lf a_nt\rf}Z_{n,i}\Rightarrow S_\nu(t),
\end{equation}
where $S_\nu$ is the L\'evy subordinator with L\'evy measure
$\nu$ and zero drift. Convergence holds weakly on the space
$D([0,\infty))$ equipped with the Skorokhod $J_1$-topology.
\end{theorem}

\begin{remark*}
Condition (\ref{careful.20}) ensures that ``small'' terms in the sum
do not
contribute to the limit. It is almost a consequence of assumption (\ref{dr.2})
and the hypothesis on the limiting measure $\nu$. However, in the general
context of triangular arrays, one can easily construct counterexamples if
(\ref{careful.20}) is not imposed.
\end{remark*}

\begin{remark*}
We emphasize that the result holds in the (usual) $J_1$-topology, since this
is crucial for applications to correlation functions. See~\cite{Whi02} for
an extensive discussion of topologies on \cadlag spaces.
\end{remark*}

The straightforward idea is to apply this theorem with
$Z_{n,i}\equiv\break c_n^{-1}{\lambda}_n^{-1}(J_n(i))e_{n,i}$. This was done in
\cite{G10a}
(see Theorem 1.3.) and applied to the case of Bouchaud's trap models
\cite{G10a} and in the
random energy model~\cite{G10b,G10c} where it allowed the author to extend
all previously know results in a very elegant way.

In models with strong local correlations, such as the $p$-spin SK model,
one cannot, however, expect that with this choice the
conditions of the theorem will be satisfied. In fact, one easily convinces
oneself that contributions to the sum in (\ref{dr.5}) cannot only come from
singly widely separated points $i$, but that such contributing terms
form clusters due to the correlations.

In this paper we show that a good way to proceed in such a situation
is to
use a suitable blocking. Introduce a new scale, $\th_n$,
and use Theorem~\ref{dr.1}
with the random variables
\begin{equation}\label{dr.6}
Z_{n,i}\equiv\sum_{j=\th_n (i-1)+1}^{\th_ni}
c_n^{-1}{\lambda}_n^{-1}(J_n(i))e_{n,i},\qquad i\geq1.
\end{equation}
The purpose of this procedure is that if $J_n$ is rapidly
mixing, we can hope to choose $\th_n\ll a_n$ such that the random variables
$J_n(\th_n i), i\in\N$ are close to independent and distributed according
to the invariant distribution $\pi_n$. But then, under the law $\PP
_{\mu_n}$,
also the random variables $Z_{n,i}$ are close to independent and identically
distributed (although with a complicated distribution, that is, a
random variable
depending on the random environment). That should put us in a position to
verify
the conditions of Theorem~\ref{dr.1}.

Let us now look at this in more detail.

For $y\in\VV_n$ and $u>0$, let
\begin{equation}
Q^{u}_n(y)\equiv\PP_{y}\Biggl(
\sum_{j=0}^{\theta_n-1}{\lambda}_n^{-1}(J_n(j))e_{n,j}>c_n u
\Biggr)
\label{2.1}
\end{equation}
be the tail distribution of the aggregated jumps when $X_n$ starts in $y$.
Note that $Q^{u}_n(y)$, $y\in\VV_n$, is a random function on the
probability space $(\O, \FF, \P)$,
and so is the function $F^{u}_n(y)$, $y\in\VV_n$ defined through
\begin{equation}
F^{u}_n(y)\equiv\sum_{x\in\VV_n}p_n(y,x)Q^{u}_n(x) .
\label{2.2}
\end{equation}
Writing
$k_n(t)\equiv\lf{\lf a_n t\rf}/{\theta_n}\rf$,
we further define
\begin{eqnarray}\label{2.3.0}
\nu_n^{J,t}(u,\infty)
&\equiv& \sum_{i=0}^{k_n(t)-1}F^{u}_n\bigl(J_n(\theta_n(i))\bigr) ,
\\
({\sigma}_n^{J,t})^2(u,\infty)
&\equiv&
\sum_{i=0}^{k_n(t)-1}\bigl[
F^{u}_n\bigl(J_n(\theta_n(i))\bigr)
\bigr]^2 .
\label{2.3}
\end{eqnarray}
Finally, we set
\begin{equation}\label{blocked.0}
\quad\bar S_n(k)\equiv\sum_{i=1}^{k}
\Biggl(\sum_{j=\th_n(i-1)+1}^{\theta_n i}c_n^{-1}{\lambda}
_n^{-1}(J_n(j))e_{n,j}\Biggr)+
c_n^{-1}{\lambda}_n^{-1}(J_n(0))e_{n,0}
\end{equation}
and
\begin{equation}\label{blocked.1}
S_n^b(t)\equiv\bar S_n(k_n(t)).
\end{equation}

We now formulate four conditions for the sequence
$S_n$ to converge to a subordinator.
Note that these conditions refer to given sequences of numbers $a_n,
c_n$ and~$\th_n$ as well as a given realization of the random environment.

\renewcommand{\thecondition}{(A\arabic{condition})}
\begin{condition}\label{coa1}
There exists a ${\sigma}$-finite measure $\nu$ on $(0,\infty)$ satisfying
the hypothesis stated in Theorem~\ref{dr.1},
and such that for all $t>0$ and all $u>0$,
\begin{equation}
P_{\mu_n}^{}\bigl(
|
\nu_n^{J,t}(u,\infty)-t\nu(u,\infty)
|
<{\varepsilon}
\bigr)=1-o(1) \qquad\forall{\varepsilon}>0 .
\label{2.A1}
\end{equation}
\end{condition}

\begin{condition}\label{coa2} For all $u>0$ and all $t>0$,
\begin{equation}
P_{\mu_n}^{}\bigl(
({\sigma}_n^{J,t})^2(u,\infty)<{\varepsilon}
\bigr)=1-o(1)  \qquad\forall{\varepsilon}>0 .
\label{2.A2}
\end{equation}
\end{condition}

\begin{condition}\label{coa3} For all $t>0$,
\begin{equation}\label{careful.30}
\lim_{{\varepsilon}\downarrow0}\limsup_{n\uparrow\infty}
\EE_{\mu_n}\sum_{i=1}^{\lf a_n t\rf} \1_{\{{\lambda}
_n^{-1}(J_n(i))e_i\leq
c_n{\varepsilon}\}} c_n^{-1}{\lambda}_n^{-1}(J_n(i))e_i=0.
\end{equation}
\end{condition}

\renewcommand{\thecondition}{(A\arabic{condition}$'$)}
\setcounter{condition}{-1}
\begin{condition}\label{coa0'} For all $v>0$,
\begin{equation}
\sum_{x\in\VV_n}\mu_n^{}(x)e^{-vc_n{\lambda}^{}_n(x)}=o(1) .
\label{2.A0'}
\end{equation}
\end{condition}

\begin{theorem}\label{2.theo1}
For all sequences of initial distributions $\mu_n$
and all sequences~$a_n$, $c_n$ and $1\leq\theta_n\ll a_n$,
for which Conditions~\ref{coa0'},~\ref{coa1},~\ref{coa2} and \ref
{coa3} are verified,
either $\P$-almost surely or in $\P$-probability [meaning that the
terms $o(1)$ converge to zero either almost surely or in probability, resp.],
the following holds w.r.t. the same convergence mode:
\begin{equation}
S^b_n(\cdot)\Rightarrow S_\nu(\cdot),
\label{2.theo1.1}
\end{equation}
where $S_\nu$ is the L\'evy subordinator with L\'evy measure
$\nu$ and zero drift. Convergence holds weakly on the space
$D([0,\infty))$ equipped with the Skorokhod $J_1$-topology.
\end{theorem}

\begin{remark*} Note that Condition~\ref{coa0'} is there to ensure
that last
term in
(\ref{blocked.0}) converges to zero in the limit $n\uparrow\infty$.
\end{remark*}

\begin{remark*} The result of this theorem is stated for the \textit{blocked}
process $S^b_n(t)$. It implies immediately that under the same hypothesis,
the original process $S_n(t)$ [defined in (\ref{1.clock})] converges to
$S_\nu$
in the weaker $M_1$-topology; see~\cite{Whi02} for a detailed
discussion of
Skorokhod topologies. However, the statement of the theorem is
strictly stronger than just convergence in $M_1$, and it is this form
that is useful in applications.
\end{remark*}

\begin{remark*}
To extract detailed information on the process $X_n$, for example the
behavior of
correlation functions, from the convergence of the blocked clock process,
one needs further information on the typical behavior of the process
during the $\th_n$ steps of a single block. This is a model-dependent issue,
and we will exemplify how this can be done in the context of the
$p$-psin SK model.
\end{remark*}

We now come to the key step in our argument. This consists in reducing
Conditions~\ref{coa1} and~\ref{coa2} of Theorem~\ref{2.theo1} to:
(i) a
\textit{mixing condition} for the
chain $J_n$ and (ii) a \textit{law of large numbers} for the random variables
$Q_n$.\vadjust{\goodbreak}

Again we formulate three conditions for given sequences $a_n, c_n$ and a
given realization of the random environment.

\renewcommand{\thecondition}{(A\arabic{condition}-1)}
\setcounter{condition}{0}
\begin{condition}\label{coa1.1} Let $J_n$ be a periodic Markov chain
with period $q$.
There exists an integer sequence $\ell_n\in\N$, and a positive
decreasing sequence $\rho_n$, satisfying
$\rho_n\downarrow0$ as $n\uparrow\infty$, such that for all pairs
$x,y\in\VV_n$, and all $i\geq0$,
\begin{equation}
\quad\sum_{k=0}^{q-1}P_{\pi_n}\bigl(J_n(i+\ell_n+k)=y, J_n(i)=x
\bigr)\leq
(1+\rho_n)\pi_n(x)\pi_n(y) .
\label{3.A1-1}
\end{equation}
\end{condition}

\begin{condition}\label{coa2.1}
There exists a measure $\nu$, as in condition~\ref{coa1}, such that
\begin{equation}\label{3.prop1.0}
\nu_n^{t}(u,\infty)\equiv k_n(t)\sum_{x\in\VV_n}\pi_n(x)Q^{u}_n(x)
\rightarrow t\nu(u,\infty) ,
\end{equation}
and
\begin{equation}\label{3.prop1.01}\qquad
({\sigma}_n^{t})^2(u,\infty)\equiv k_n(t)\sum_{x\in\VV_n}\sum
_{x'\in
\VV_n}\pi
_n(x)p_n^{(2)}(x,x')Q^{u}_n(x)Q^{u}_n(x')\rightarrow0 ,
\end{equation}
where
$p_n^{(2)}(x,x')
=\sum_{y\in\VV_n}p_n(x,y)p_n(y,x')
$
are 2-step transition probabilities.
\end{condition}

\begin{condition}\label{coa3.1}
For all $t>0$,
\begin{equation}\label{careful.30}
\lim_{{\varepsilon}\downarrow0}\limsup_{n\uparrow\infty} \lf a_n
t\rf
\EE_{\pi_n} \1_{\{{\lambda}_n^{-1}(J_n(0))e_0\leq c_n{\varepsilon}\}
} c_n^{-1}{\lambda}
_n^{-1}(J_n(0))e_0=0.
\end{equation}
\end{condition}

\begin{remark*} The limiting measure $\nu$ may be deterministic or random.
\end{remark*}

\begin{theorem}\label{main.1}
Assume that for $\mu_n=\pi_n$ and for sequences $a_n$, $c_n$, $\ell_n$
and $\ell_n\leq\theta_n\ll a_n$,
Conditions~\ref{coa1.1},~\ref{coa2.1},~\ref{coa3.1} and~\ref{coa0'}
hold $\P$-a.s., respectively in $\P$-probability. Then
the sequence of random stochastic process $S^b_n$ converges to
the process $S_\nu$, weakly in the Skorokhod space $D[0,\infty)$ equipped
with the $J_1$-topology, $\P$-almost surely, respectively in $\P$-probability.
\end{theorem}

\subsection{Application to the $p$-spin SK model}
Theorem~\ref{main.1} is the central result of this paper. It provides
a very nice tool to prove convergence results of clock processes
almost surely with respect to the random environment, that is the physically
desirable mode. It is capable of dealing with correlations that have an
effect, such as are present in the $p$-spin SK model. In this model,
the underlying graphs $\VV_n$ are the hypercubes $\S_n=\{-1,1\}^n$.
On $\S_n$ we consider a Gaussian process, $H_n$, with zero mean and covariance
\begin{equation}\label{pspin.1}
\E H_n(x)H_n(x') =nR_n(x,x')^p,
\end{equation}
where $R_n(x,x')\equiv\frac1n \sum_{i=1}^n x_ix'_i$. The random
environment, $\t_n(x)$, is then defined in terms of $H_n$ by
\begin{equation}\label{pspin.2}
\t_n(x)\equiv\exp(\b H_n(x)),\vadjust{\goodbreak}
\end{equation}
with $\b\in\R_+$ the inverse temperature. The Markov chain, $J_n$, is
chosen as the simple random walk on $\S_n$, that is,
\begin{equation}\label{pspin.3}
p_n(x,x')=
\cases{\displaystyle\frac1n, &\quad$\mbox{if } \dist(x,x')=1,$\vspace*{2pt}\cr
0,&\quad$\mbox{else};$}
\end{equation}
here $\dist(\cdot,\cdot)$ is the graph distance on $\S_n$,
\begin{equation}\label{graph.dist}
\dist(x,x')\equiv\frac12 \sum_{i=1}^n |x_i-x'_i|.
\end{equation}
This chain has for unique invariant measure the measure $\pi_n(x)=2^{-n}$.
Finally, choosing $C=2^n$ in (\ref{inv-ab.1}), the mean holding times,
${\lambda}
^{-1}_n(x)$,
reduce to ${\lambda}^{-1}_n(x)= \t_n(x)$.

\begin{theorem}
\label{p:main} For any $p\geq3$,
there exists a constant $K_p>0$ that depends on $\beta$ and $\gamma$,
and a function $\zeta(p)$, such that
for all $\gamma$ satisfying
%
%
\begin{equation}
0 < \gamma< \min(\beta^2,\zeta(p) \beta),
\end{equation}
the law of
the stochastic process

%
%
\begin{equation}
\label{e:barSdef}
S^b_n(t)\equiv
e^{-\gamma n}
S_n(\th_n\lfloor t n^{1/2}e^{n\gamma^2/2\beta^2}\th
_n^{-1}\rfloor
),\qquad
t\ge0,
\end{equation}
with $\th_n=\frac{3\ln2}2 n^2$,
defined on the space of \cadlag functions equipped with the
Skorokhod $J_1$-topology, converges to the law of the
stable subordinator
$V_{\gamma/\beta^2}(t), t\ge0$, of L\'evy measure
$ K_p(\gamma/\beta^2) x^{-\gamma/\beta^2-1}\,dx$.
Convergence holds $\P$-a.s. if $p>4$, and in $\P$-probability, if
$p=3,4$.

The function $\zeta(p)$ is increasing, and it satisfies
%
%
\begin{equation}
\label{e:zeta}
\zeta(3)\simeq1.0291
\quad\mbox{and}\quad
\lim_{p\to\infty}\zeta(p)=\sqrt{2\log2}.
\end{equation}
\end{theorem}

\begin{remark*} This result implies the weaker statement that
%
%
\begin{equation}
\label{e:barSdef.1}
S_n(t)\equiv
e^{-\gamma N}
S_n(\lfloor t n^{1/2}e^{n\gamma^2/2\beta^2}\rfloor),\qquad
t\ge0,
\end{equation}
converges in the same way in the $M_1$-topology.
\end{remark*}

In~\cite{BBC08} an analogous result is proven, with the same constants
$\zeta(p)$ and $K_p$, but convergence there is in law with respect to
the random
environment (and almost sure with respect to the trajectories $J_n$).
Being able to
obtain convergence under the law of the trajectories for fixed environments,
as we do here, is a considerable conceptual
improvement.

Finally, one must ask whether the convergence of the clock process in the
form obtained here is useful for deriving aging information in the
sense that
we can control the behavior of certain correlation functions. One may be
worried that a jump in limit of the coarse-grained clock process refers to
a period of time during which the process still may make $n^2$ steps,
and our limit result tells us nothing about how the process moves
during that time.
We will, however, show that essentially all this time is spent in
a single visit to a quite small ``trap,'' within which the process does
not make more than $o(n)$
steps.

In this way we prove the almost-sure (or in probability) version of
Theorem 1.2 of~\cite{BBC08}.

\begin{theorem}
\label{t:aging}
Let $A_n^\varepsilon(t,s)$ be the event defined by
%
%
\begin{equation}
A_n^\varepsilon(t,s)=
\bigl\{
R_n\bigl(
X_n(te^{\gamma n}),
X_n\bigl((t+s)e^{\gamma n}\bigr)\bigr)
\ge1-\varepsilon\bigr\}.
\end{equation}
Then, under the hypothesis of Theorem~\ref{p:main}, for all
$\varepsilon\in(0,1)$, $t>0$ and $s>0$,
%
%
\begin{equation}
\lim_{N\to\infty}\PP_{\pi_n}(
A_n^\varepsilon(t,s))=
\frac{\sin\alpha\pi}{\pi}
\int_0^{t/(t+s )} u^{\alpha-1}(1-u)^{-\alpha} \,d u.
\end{equation}
Convergence holds $\P$-a.s. if $p>4$, and in $\P$-probability, if
$p=3,4$.
\end{theorem}

The remainder of the paper is organised as follows. In the next section we
prove Theorems~\ref{2.theo1} and~\ref{main.1}.
In Section~\ref{sec3} we apply our main theorem to the $p$-spin SK model
and prove Theorem~\ref{t:aging}.

\section{Proof of the main theorems}\label{sec2}

We now prove our main theorem. The first step is the proof of
Theorem~\ref{2.theo1}.

\subsection{\texorpdfstring{Proof of Theorem \protect\ref{2.theo1}}{Proof of Theorem 1.2}}
\mbox{}
\begin{pf} Throughout we fix a realization ${\omega}\in\O$ of the random
environment but do not make this explicit in the notation.
We set
\begin{equation}
\wh S^b_n(t)\equiv S_n^b(t)-
c_n^{-1}\lambda_n^{-1}(J_n(0))e_{n,0}.
\label{2.4}
\end{equation}
Condition~\ref{coa0'} ensures that $S^b_n -\wh S_n^b$ converges to
zero, uniformly.
Thus we must show that under Conditions~\ref{coa1} and~\ref{coa2},
\begin{equation}
\wh S^b_n(\cdot)\Rightarrow S_\nu(\cdot) .
\label{2.7}
\end{equation}
This will be a simple corollary of Theorem~\ref{dr.1}.
Recall that
\begin{equation}
k_n(t)\equiv\bigl\lf\lf a_n t\rf/\theta_n\bigr\rf,
\label{2.8}
\end{equation}
and for $i\geq1$, define
\begin{equation}
Z_{n,i} \equiv\sum_{j=\theta_n(i-1)+1}^{\theta_ni}c_n^{-1}{\lambda}
_n^{-1}(J_n(j))e_{n,j} .
\label{2.9}
\end{equation}
By (\ref{blocked.1}) and (\ref{2.4}), $\wh S^b_n(t)= \sum
_{i=1}^{k_n(t)}Z_{n,i}$.
We now want to apply Theorem~\ref{dr.1} to the latter partial sum process.
For this let $\{\FF_{n,i}, n\geq1, i\geq0\}$
be the array of sub-sigma fields of $\FF^X$ defined by (with obvious notation)
$\FF_{n,i}
={\sigma}(\bigcup_{j\leq\theta_ni} \{J_n(j),e_{n,j}\}
)$,
for $i\geq0$.
Clearly, for each $n$ and $i\geq1$, $Z_{n,i}$ is $\FF_{n,i}$
measurable and $\FF_{n,i-1}\subset\FF_{n,i}$.
Next observe that
\begin{eqnarray}\label{2.12}
&&\PP_{\mu_n}(Z_{n,i}>z | \FF_{n,i-1})
\nonumber
\\[-8pt]
\\[-8pt]
\nonumber
&&\qquad=\sum_{x\in\VV_n}\PP_{\mu_n}\bigl(J_n\bigl(\theta_n(i-1)+1\bigr)=x,
Z_{n,i}>z | \FF_{n,i-1}\bigr),
\end{eqnarray}
where
\begin{eqnarray}\label{2.13}
&&\PP_{\mu_n}\bigl(J_n\bigl(\theta_n(i-1)+1\bigr)=x, Z_{n,i}>z
|
\FF_{n,i-1}\bigr)
\nonumber
\\[-8pt]
\\[-8pt]
\nonumber
&&\qquad=\PP_{\mu_n}\bigl(J_n\bigl(\theta_n(i-1)+1\bigr)=x, Z_{n,i}>z
| J_n\bigl(\theta_n(i-1)\bigr)\bigr).
\\\nonumber
\end{eqnarray}
Using Bayes' theorem and the Markov property, the last line can be
written as
\begin{equation}
 p_n\bigl(J_n\bigl(\theta_n(i-1)\bigr),x\bigr)
\PP_{\mu_n}\Biggl(\sum_{j=1}^{\theta_n}c_n^{-1}{\lambda}
_n^{-1}\bigl(J_n(j-1)\bigr)e_{n,j-1}>z | J_n(0)=x\Biggr).
\label{2.13bis}\hspace*{-35pt}
\end{equation}
Thus, in view of (\ref{2.1}), (\ref{2.2}), (\ref{2.3.0}) and (\ref
{2.3}), it
follows from (\ref{2.12}), (\ref{2.13}) and~(\ref{2.13bis}) that
\begin{eqnarray}\label{2.14}
\sum_{i=1}^{k_n(t)}
\PP_{\mu_n}(Z_{n,i}>z \vert \FF_{n,i-1})
&=&
\sum_{i=1}^{k_n(t)}
\sum_{x\in\VV_n}p_n\bigl(J_n\bigl(\theta_n(i-1)\bigr),x\bigr)Q^{u}_n(x)
\nonumber\\
&=&
\sum_{i=1}^{k_n(t)}F^{u}_n\bigl(J_n\bigl(\theta_n(i-1)\bigr)\bigr)
\\
&=&
\nu_n^{J,t}(u,\infty)
.\nonumber
\end{eqnarray}
Similarly we get
\begin{eqnarray}\label{2.15}
\sum_{i=1}^{k_n(t)}
[\PP_{\mu_n}(Z_{n,i}>{\varepsilon}\vert \FF_{n,i-1})]^2
&=&
\sum_{i=1}^{k_n(t)}\bigl[
F^{u}_n\bigl(J_n\bigl(\theta_n(i-1)\bigr)\bigr)
\bigr]^2
\nonumber
\\[-8pt]
\\[-8pt]
\nonumber
&=&({\sigma}_n^{J,t})^2(u,\infty) .
\end{eqnarray}
From (\ref{2.14}) and (\ref{2.15}) it follows that Conditions \ref
{coa2} and~\ref{coa1} of
Theorem~\ref{2.theo1} are exactly the conditions from Theorem~\ref{dr.1}.
Similarly Condition~\ref{coa3} is Condition~\ref{careful.20}.
Therefore the conditions of Theorem~\ref{dr.1} are verified, and so
$
\wh S^b_n\Rightarrow S_\nu
$
in $D([0,\infty))$ where $S_\nu$ is a subordinator with L\'evy measure
$\nu$ and zero drift.
\end{pf}

\subsection{\texorpdfstring{Proof of Theorem \protect\ref{main.1}}{Proof of Theorem 1.3}}
The proof of Theorem~\ref{main.1} comes in two steps. In the first we
use the ergodic properties of the chain $J_n$ to pass from sums along
a chain $J_n$ to averages with respect to the invariant measure of $J_n$.

We assume from now on that the initial distribution $\mu_n$ is the
invariant measure~$\pi_n$ of the jump chain $J_n$.

\begin{proposition}{\label{3.prop1}} Let $\mu_n=\pi_n$. Assume that
Condition~\ref{coa1.1} is satisfied. Then, choosing $\theta_n\geq\ell_n$,
the following holds:
for all $t>0$ and all $u>0$, we have that for all ${\varepsilon}>0$,
\begin{eqnarray}\label{3.prop1.1}
&& P_{\pi_n}\bigl(|\nu_n^{J,t}(u,\infty)-\nu_n^{t}(u,\infty
)
|
\geq{\varepsilon}\bigr)
\nonumber
\\[-8pt]
\\[-8pt]
\nonumber
&&\qquad\leq
{\varepsilon}^{-2}[\rho_n(\nu_n^{t}(u,\infty))^2+({\sigma}
_n^{t})^2(u,\infty)] ,
\end{eqnarray}
and
\begin{equation}
P_{\pi_n}\bigl(({\sigma}_n^{J,t})^2(u,\infty)\geq{\varepsilon}\bigr)
\leq
{{\varepsilon}}^{-1}({\sigma}_n^{t})^2(u,\infty) .
\label{3.prop1.2}
\end{equation}
\end{proposition}

\begin{pf} To simplify notation, we only give the proof for the
case when the chain $J_n$ is aperiodic, that is, $q=1$. Details of how to
deal with the general periodic case can be found in the proof of
Proposition 4.1 of~\cite{G10a}.

Let us first establish that
\begin{eqnarray}\label{3.prop1.8.0}
E_{\pi_n}[\nu_n^{J,t}(y)]&=&\nu_n^{t}(u,\infty) ,
\\
E_{\pi_n}[({\sigma}_n^{J,t})^2(u,\infty)]
&=&({\sigma}_n^{t})^2(u,\infty) .
\label{3.prop1.8}
\end{eqnarray}
To this end set
\begin{equation}
\pi_n^{J,t}(x)={k^{-1}_n(t)}\sum_{j=1}^{k_n(t)}\1_{\{J_n(\theta
_n(j-1))=x\}} ,\qquad x\in\VV_n .
\label{3.prop1.4}
\end{equation}
Then, equations (\ref{2.3.0}) and (\ref{2.3}) may be rewritten as
\begin{eqnarray}
\nu_n^{J,t}(u,\infty)
&=&k_n(t)\sum_{y\in\VV_n}\pi_n^{J,t}(y)F^{u}_n(y) ,
\\
({\sigma}_n^{J,t})^2(u,\infty)
&=&
k_n(t)\sum_{y\in\VV_n}\pi_n^{J,t}(y)(F^{u}_n(y))^2 .
\label{3.prop1.5}
\end{eqnarray}
Since by assumption the initial distribution is the invariant measure
$\pi_n$ of $J_n$,
the chain variables $(J_n(j), j\geq1)$ satisfy
$
P_{\pi_n}(J_n(j)=x)=\pi_n(x)
$
for all $x\in\VV_n$, and all $j\geq1$. Hence
\begin{eqnarray}\label{3.prop1.7}
E_{\pi_n}[\pi_n^{J,t}(y)]&=&\pi_n(y) ,
\\\label{3.prop1.7'}
E_{\pi_n}[\nu_n^{J,t}(u,\infty)]&=&k_n(t)\sum_{x\in
\VV_n}\pi
_n(x)F^{u}_n(x) ,
\\\label{3.prop1.7''}
E_{\pi_n}[({\sigma}_n^{J,t})^2(u,\infty)]&=&k_n(t)\sum_{x\in
\VV
_n}\pi_n(x)(F^{u}_n(x))^2 ,
\end{eqnarray}
and equations (\ref{3.prop1.8.0}) and (\ref{3.prop1.8}) now follow readily
from these identities.
Indeed, inserting (\ref{2.2}) into (\ref{3.prop1.7'})
and using that $\pi_n$ is the invariant measure of $J_n$, we get
\begin{eqnarray}\label{3.prop1.07}
E_{\pi_n}[\nu_n^{J,t}(u,\infty)]
&=&k_n(t)\sum_{y\in\VV_n}\sum_{x\in\VV_n}\pi
_n(x)p_n(x,y)Q^{u}_n(y) ,
\\
&=&k_n(t)\sum_{y\in\VV_n}\pi_n(y)Q^{u}_n(y) ,
\end{eqnarray}
which proves (\ref{3.prop1.8.0}).
Similarly, inserting (\ref{2.2}) into (\ref{3.prop1.7''}) yields
\begin{equation}\label{3.prop1.007}
E_{\pi_n}[({\sigma}_n^{J,t})^2(u,\infty)]
=k_n(t)\sum_{x\in\VV_n}\pi_n(x)\biggl(\sum_{y\in\VV
_n}p_n(x,y)Q^{u}_n(y)\biggr)^2 ,
\end{equation}
which gives (\ref{3.prop1.8}), once observed that, by reversibility,
$
\sum_{x\in\VV_n}\pi_n(x)p_n(x,y)\times p_n(x,y')
=\pi_n(y)\sum_{x\in\VV_n}p_n(y,x)p_n(x,y')
=\pi_n(y)p^{(2)}_n(y,y')
$.

We are now ready to prove the proposition.
In view of (\ref{3.prop1.8}), (\ref{3.prop1.2}) is nothing but a first
order Chebyshev inequality.
To establish (\ref{3.prop1.1}) set
\begin{equation}
\qquad\quad\D_{ij}(x,y)=P_{\pi_n}\bigl(J_n\bigl(\theta_n(i-1)\bigr)=x,
J_n\bigl(\theta
_n(j-1)\bigr)=y\bigr)-\pi_n(x)\pi_n(y) .
\end{equation}
A second-order Chebyshev inequality together with expressions (\ref
{3.prop1.7'})
of $E_{\pi_n}[\nu_n^{J,t}(u,\infty)]$ yield
\begin{eqnarray}\label{3.prop1.9}
&&P_{\pi_n}\bigl(|\nu_n^{J,t}(u,\infty)-E_{\pi_n}[\nu
_n^{J,t}(u,\infty)]|\geq{\varepsilon}\bigr)
\nonumber\\
&&\qquad\leq{\varepsilon}^{-2}E_{\pi_n}\biggl[k_n(t)\sum_{y\in\VV
_n}\bigl(\pi
_n^{J,t}(y)-\pi_n(y)\bigr)F^{u}_n(y)\biggr]^2
\\
&&\qquad={\varepsilon}^{-2}\sum_{x\in\VV_n}\sum_{y\in\VV
_n}F^{u}_n(x)F^{u}_n(y)\sum
_{i=1}^{k_n(t)}\sum_{j=1}^{k_n(t)}\D_{ij}(x,y) .\nonumber
\end{eqnarray}
Now
$
\sum_{i=1}^{k_n(t)}\sum_{j=1}^{k_n(t)}\D_{ij}(x,y)=(\overline
{I})+(\overline{\mathit{II}})
$
where
\begin{equation}
(\overline{I})\equiv\sum_{i=1}^{k_n(t)}\sum_{j=1}^{k_n(t)}\D
_{ij}(x,y)\1
_{\{j\neq i\}}
\leq\rho_n k^2_n(t) \pi_n(x)\pi_n(y) ,
\label{3.prop1.10}
\end{equation}
as follows from Condition (A1-1), choosing $\theta_n\geq\ell_n$, and
\begin{eqnarray}\label{3.prop1.11}
(\overline{\mathit{II}})
&\equiv&\sum_{1\leq i\leq k_n(t)}\D_{ii}(x,x)\1_{\{x=y\}}
\nonumber\\
&=&k_n(t)\bigl[P_{\pi_n}\bigl(J_n\bigl(\theta_n(i-1)\bigr)=x\bigr
)-\pi
^2_n(x)\bigr]\1_{\{x=y\}}
\\
&=&k_n(t)\pi_n(x)\bigl(1-\pi_n(x)\bigr)\1_{\{x=y\}} .\nonumber
\end{eqnarray}
Inserting (\ref{3.prop1.11}) and (\ref{3.prop1.10}) in (\ref
{3.prop1.9}) we
obtain, using again (\ref{3.prop1.8}) and (\ref{3.prop1.7}), that
\begin{eqnarray}\label{3.prop1.12}
&& P_{\pi_n}\bigl(|\nu_n^{J,t}(u,\infty)-E_{\pi_n}[\nu
_n^{J,t}(u,\infty)]|\geq{\varepsilon}\bigr)
\nonumber
\\[-8pt]
\\[-8pt]
\nonumber
&&\qquad\leq
{\varepsilon}^{-2}\bigl[\rho_n(\nu_n^{t}(u,\infty))^2+({\sigma}
_n^{t})^2(u,\infty)\bigr] .
\end{eqnarray}
Proposition~\ref{3.prop1} is proven.
\end{pf}

\begin{pf*}{Proof of Theorem \protect\ref{main.1}}
The proof of Theorem~\ref{main.1} is now immediate: combine the conclusions
of Proposition~\ref{3.prop1} with Condition~\ref{coa2.1} to get both
conditions
\ref{coa1} and~\ref{coa2}.
Finally, Condition~\ref{coa3} is Condition~\ref{coa3.1}, since we are
starting from the
invariant measure.
\end{pf*}

\section{Application to the $p$-spin SK model}\label{sec3}

In this section we show how Conditions~\ref{coa1.1} and~\ref{coa2.1} can be
verified in
the case of the random hopping time dynamics of the $p$-spin SK model.

The proof contains four steps, two of which are quite immediate.

Conditions~\ref{coa1.1} for simple random walk has been established, for
example, in
\cite{BBC08} and~\cite{G10b}. The following lemma is taken from
Proposition 3.12 of~\cite{G10b}.

\begin{lemma}\label{mixing.1} Let $P_{\pi_n}$ be the law of the simple
random walk on the hypercube $\S_n$ started in the uniform distribution.
Let $\th_n=\frac{3 \ln2}2 n^2$. Then, for any $x,y\in\S_n$ and any
$i\geq0$,
\begin{equation}\label{mixing.2}\quad
\Biggl| \sum_{k=0}^1P_{\pi_n}\bigl(J_n(\th_n+i+k)=y,J_n(0)=x\bigr)
-2\pi_n(x)\pi_n(y)\Biggr|\leq2^{-3n+1}.
\end{equation}
\end{lemma}

Clearly this implies that Condition~\ref{coa1.1} holds.

We now turn to the first part of Condition~\ref{coa2.1}. We will show that
\begin{equation}\label{pspin.10}
\nu_n^t(u,\infty)\rightarrow\nu^t(u,\infty)=t K_p u^{-\g/\b^2},
\end{equation}
almost surely, respectively, in probability, as $n\uparrow\infty$.

\subsection{Laplace transforms} Instead of proving the convergence of
the distribution functions $\nu_n^t$ directly, we pass to their
Laplace transforms, prove their convergence and then use
Feller's continuity lemma to deduce convergence of the original
objects.

For $v> 0$, consider the Laplace transforms
\begin{eqnarray}\label{0}
\hat\nu_n^{t}(v)&=&\int_{0}^{\infty} due^{-uv}\nu_n^{t}(u,\infty
),
\nonumber
\\[-8pt]
\\[-8pt]
\nonumber
\hat\nu^{t}(v)&=&\int_{0}^{\infty} due^{-uv}\nu^{t}(u,\infty).
\end{eqnarray}
With $Z_n\equiv\sum_{j=0}^{\th_n-1} c_n^{-1}{\lambda
}_n^{-1}(J_n(j))e_{n,j}$,
we have,
by definition of $\nu_n^{t}(u,\infty)$,
\[
\nu_n^{t}(u,\infty)=k_n(t)\sum_{x\in\VV_n}\pi_n(x)Q^{u}_n(x)=
k_n(t)\PP_{\pi_n}(Z_{n}>u).
\label{3}
\]
Hence
\begin{eqnarray}\label{4}
\hat\nu_n^{t}(v)&=&\int_{0}^{\infty} due^{-uv}\nu_n^{t}(u,\infty)
\nonumber\\
&=&k_n(t)\int_{0}^{\infty} due^{-uv}\PP_{\pi_n}(Z_n>u)
\\
&=&k_n(t)\frac{1-\EE_{\pi_n}(e^{-vZ_n})}{v},\nonumber
\end{eqnarray}
where the last equality follows by integration by parts.

\subsection{\texorpdfstring{Convergence of $\E\hat\nu_n^t(v)$}{Convergence of E nu n t(v)}}

The following lemma is an easy consequence of the results of~\cite{BBC08}:

\begin{lemma}\label{pspin.11} Let $c_n=e^{\g n}$, $a_n= n^{1/2} e^{n
\g
^2/2\b^2}$.
For any $p\geq3$, and $\b,\g>0$ such that $\g/\b^2\in(0,1)$, there
exists a
finite positive constant, $K_p$, such that for any $v>0$,
\begin{equation}\label{pspin.12}
\lim_{n\uparrow\infty} k_n(t) \E
[1-\EE_{\pi_n}(e^{-vZ_n})] =K_pt v^{\g/\b^2}.
\end{equation}
\end{lemma}

\begin{pf} We rely essentially on the results of~\cite{BBC08}.
In that paper the Laplace transforms $\E e^{-v Z_n}$ were computed even for
$\th_n= a_nt$. We just recall the key ideas and the main steps.

The point in~\cite{BBC08} is to first fix a realization of the chain $J_n$,
and to define, for a given realization, the one-dimensional normal
Gaussian process
\begin{equation}\label{jiri.1}
U^0(i)\equiv n^{-1/2} H_n(J_n(i)),
\end{equation}
with covariance
\begin{equation}\label{jiri.2}
\L_{ij}^0= n^{-1} \E H_n(J_n(i))H_n(J_n(j)) = R_n(J_n(i),J_n(j))^p.
\end{equation}
Moreover, they define a comparison process, $U^1$, as follows. Let $\nu
$ be an integer of order $n^\rho$, with $\rho\in(1/2,1)$. Then $U^1$
has covariance matrix
\begin{equation}\label{jiri.2-1}
\L^1_{ij}=
\cases{ 1-2pn^{-1}|i-j|,& \quad$\mbox{if } \lf i/\nu\rf=\lf
j/\nu\rf,$\vspace*{2pt}\cr
0,& \quad$\mbox{else}.$}
\end{equation}
Finally they define the interpolating family of processes, for $h\in[0,1]$,
\begin{equation}\label{jiri.3}
U^h(i)\equiv\sqrt h U^1(i)+\sqrt{1-h} U^0(i).
\end{equation}
For any normal Gaussian process, $U$, indexed by $\N$, define the functions
\begin{equation}\label{jiri.4}
F_n(U,v,k)\equiv\exp\Biggl(-v c_n^{-1}\sum_{i=0}^{k-1}
e_{n,i} e^{\b\sqrt n U_i}
\Biggr)\vadjust{\goodbreak}
\end{equation}
and
\begin{equation}
\label{jiri.5.1}\quad
\quad\EE_{\pi_n}( F(U,v,k)\vert \FF^J)
\equiv G(U,v,k)= \exp\Biggl(-\sum_{i=0}^{k-1} g(vc_n^{-1} e^{\b
\sqrt n U_i}
)\Biggr),
\end{equation}
with $g(x)=\ln(1+x)$.

Then the Laplace transforms we are after can be written as
\begin{eqnarray}\label{jiri.5}
\E\EE_{\pi_n} e^{-vZ_n}&=&
\E\EE_{\pi_n} ( \EE_{\pi_n}( e^{-vZ_n}\vert \FF
^J))
\nonumber
\\[-8pt]
\\[-8pt]
\nonumber
&=& E_{\pi_n} \E G(U^0,v,\th_n).
\end{eqnarray}
Here we used that the conditional expectation, given $\FF^J$, is just the
expectation with
respect to the variables $e_{n,i}$, which can be computed explicitly,
and gives
rise to the function $G$.

The idea is now that $U^1$ is a good enough approximation to $U^0$, for most
realizations of the chain $J$, to allow us to replace $U^0$ by $U^1$ in the
last line above.

More precisely, we have the following estimate.

\begin{lemma}\label{jiri.6}
With the notation above we have that for all $p\geq3$,
\begin{equation}\label{jiri.7}
k_n(t) E_{\pi_n} | \E G(U^0,v,\th_n)
- \E G(U^1,v,\th_n)
|\leq tC n^{1/2}/\nu.
\end{equation}
\end{lemma}

\begin{remark*} In~\cite{BBC08} (see Proposition 3.1) it is proven
that $E_{\pi_n}$-almost surely,
\begin{equation}\label{jiri.7.1}
\E G(U^0,v,\lf a_n t\rf)
- \E G(U^1,v,\lf a_nt\rf)
\rightarrow0.
\end{equation}
This result would not be expected for our expression, but we do not
need this.
The proof of Proposition 3.1 of~\cite{BBC08}, however, directly
implies our
Lemma~\ref{jiri.6}.
\end{remark*}

The computation of the expression involving the comparison process
$U^1$ is
fairly easy. First, note that by independence (and making for
simplicity the
assumption that $\th_n$ is an integer multiple of ${\nu}$),
\begin{eqnarray}\label{jiri.8}
\E G(U^1,v,\th_n)
&=& [ \E G(U^1,v,\nu)]^{\th_n/\nu}
\nonumber
\\[-8pt]
\\[-8pt]
\nonumber
&=&
\bigl[1-\bigl(1- \E G(U^1,v,\nu)\bigr)
\bigr]^{\th_n/\nu}.
\end{eqnarray}
But in~\cite{BBC08}, Proposition 2.1, it is shown that
\begin{equation}\label{jiri.9}
a_n\nu^{-1} \bigl(1- \E G(U^1,v,\nu)\bigr) \rightarrow K_pv^{\g
/\b^2}.
\end{equation}
This implies immediately that
\begin{equation}\label{jiri.10}
k_n(t) \bigl\{ 1-\bigl[1-\bigl(1- \E G(U^1,v,\nu)\bigr)
\bigr]^{\th
_n/\nu}\bigr\}
\rightarrow K_pv^{\g/\b^2}t,
\end{equation}
as desired. Combining this with Lemma~\ref{jiri.6}, the assertion of Lemma
\ref{pspin.11} follows.
\end{pf}

\subsection{\texorpdfstring{Concentration of $\nu_n^t$}{Concentration of nu n t}}
To complete the proof, we need to control the fluctuations of $\nu_n^t$.

\begin{lemma}\label{jiri.11}
Under the same hypothesis as in Lemma~\ref{pspin.11},
there exists an increasing function,
$\zeta(p)$, such that for all $p\geq3$, $\zeta(p)>1$, and $\zeta
(p)\uparrow
\sqrt{2\ln2}$, such that, if $\g/\b^2<\min(1,\zeta(p)/\b)$,
\begin{equation}\label{jiri.12}
\E\bigl(\hat\nu_n^t(v)-\E\hat\nu_n(v)\bigr)^{2}\leq C n^{1-p/2}.
\end{equation}
\end{lemma}

\begin{pf} The proof is again very similar to the proof of Proposition
3.1 in~\cite{BBC08}. We have to compute
\begin{equation}\label{jiri.12.1}
\E(\EE_{\pi_n} e^{-vZ_n})^2
=E_{\pi_n}E'_{\pi_n}\bigl(\E\EE_{\pi_n}\EE_{\pi_n}'
\bigl(e^{-v(Z_n+Z_n')}\vert \FF^{J}\times\FF^{J'}\bigr)\bigr),
\end{equation}
where\vspace*{1pt} $Z'_n\equiv\sum_{j=0}^{\th_n-1} c_n^{-1}{\lambda}
_n^{-1}(J'_n(j))e'_{n,j}$, $J'_n$ and $(e'_{n,i} ,n\in\N, i\in\N)$
being, respectively, independent copies of $J_n$ and $(e_{n,i} ,n\in
\N
, i\in\N)$.
To express this as in the previous proof, we introduce
the Gaussian process $V^0$ by
\begin{equation}
\label{jiri.13}
V^0(i)\equiv
\cases{ n^{-1/2} H_n(J_n(i)),& \quad$\mbox{if } 0\leq i\leq
\th_{n}-1,$\vspace*{2pt}\cr
n^{-1/2} H_n(J'_n(i)),& \quad$\mbox{if } \th_n\leq i\leq
2\th_{n}-1.$}
\end{equation}
Then, with the notation of (\ref{jiri.5.1}),
\begin{equation}\label{jiri.14}
\EE_{\pi_n}\EE'_{\pi_n}
\bigl(e^{-v(Z_n+Z_n')}\vert \FF^{J}\times\FF^{J'}\bigr)
=G(V^0,v,2\th_n).
\end{equation}
Next we define the comparison process
$V^1$ with covariance matrix
\begin{equation}\label{jiri.15}
\L^{2}_{ij}\equiv
\cases{ \L^0_{ij},& \quad$\mbox{if }
\max(i, j)< \th_n \mbox{ or } \min(i,j) \geq\th_n,$\vspace
*{2pt}\cr
0,& \quad$\mbox{else}.$}
\end{equation}
The point is that
\begin{equation}\label{jiri.16}
\quad E_{\pi_n}E'_{\pi_n}\E G(V^1,v,2\th_n) = (E_{\pi_n} \E
G(V^0,v,\th
_n))^2=
( \E\EE_{\pi_n} e^{-vZ_n})^2.
\end{equation}
On the other hand, using the standard Gaussian interpolation formula,
we obtain the representation
\begin{eqnarray}\label{jiri.17}
&&\E G(V^1,v,2\th)-\E G(V^0,v,2\th)
\nonumber
\\[-8pt]
\\[-8pt]
\nonumber
&&\qquad=
\frac12 \int_0^1 \mathop{\sum_{0\leq i<\th_n}}_{\th_n\leq j<2\th
_n} \L_{ij}^0 \E\frac{\del^2 G(V^h,v,2\th_n)}
{\del v_i\,\del v_j}\,dh + ( i\leftrightarrow j),
\end{eqnarray}
where the interpolating process $V^h$ is defined analogously to (\ref{jiri.3}).
The second derivatives of $G$ were computed and bounded in \cite
{BBC08} [see equation (3.7) and Lemma~3.2].
We recall the following bounds:

\begin{lemma}\label{jiri.19} With the notation above and the
assumptions of
Lemma~\ref{pspin.11},
\begin{eqnarray}\label{jiri.20}
&&\E\biggl| \frac{\del^2 G(V^h,v,2\th_n)}
{\del v_i\,\del v_j}\biggr|\nonumber\\
&&\qquad\leq v^2c_n^{-2} \b^2n \E\bigl[ e^{\b\sqrt n
(V^h(i)+V^h(j))} \exp\bigl(-2 g\bigl(c_n^{-1}ve^{\b\sqrt n
V^h(i)}\bigr)
\nonumber
\\[-8pt]
\\[-8pt]
\nonumber
&&\hspace*{162pt}\qquad{}-2 g\bigl(c_n^{-1}ve^{\b\sqrt n V^h(j)}\bigr)\bigr)\bigr]
\\
&&\qquad\equiv\Xi_n(\L^h_{ij}).\nonumber
\end{eqnarray}
Moreover, for ${\lambda}>0$ small enough,
%
%
\begin{equation}
\label{jiri.21}
\Xi_n(c)\leq\bar\Xi_n (c)=
\cases{
C
\bigl({(1-c)^{-1/2}} \land\sqrt n\bigr)
e^{-{(\gamma^2n)}/{(\beta^2(1+c))}},\vspace*{2pt}\cr
\qquad$\mbox{if $1>c>\gamma/ \beta^{2}+\lambda-1 $,}$
\vspace*{2pt}\cr
C n
e^{-n(\beta^2(1+c)-2\gamma)},\vspace*{2pt}\cr
\qquad$\mbox{if $c\le(\gamma/ \beta^{2})+\lambda-1 $,}$}
\end{equation}
where $C(\gamma,\beta,v,\lambda)$
is a suitably chosen constant independent of $n$ and $c$.
\end{lemma}

\begin{remark*}
Notice that, since $\g/\b^2<1$ under our hypothesis, we can always
choose ${\lambda}$ such that the top line in (\ref{jiri.21})
covers the case $c\geq0$.
\end{remark*}

Note that for $c\geq0$
(see equation (3.25) in~\cite{BBC08}),
\begin{equation}\label{jiri.22}
\int_0^{1} \Xi_n\bigl((1-h)c\bigr)\,dh\leq2 C \exp\biggl(-\frac{\gamma
^2n}{\beta^2(1+c)}
\biggr).
\end{equation}
The terms with negative correlation are in principle
smaller than those with positive one, but some thought reveals that one
cannot really gain substantially over the bound
\begin{equation}\label{jiri.23}
\int_{0}^{1} \Xi_n\bigl((1-h)c\bigr)\,dh\leq
C \exp\biggl(-\frac{\gamma^2n}{\beta^2}\biggr),
\end{equation}
that is, used in~\cite{BBC08} [see equation (3.24)].

Next we must compute the probability that $ \L_{ij}^0$ takes on a specific
value. But since $ \L_{ij}^0$ is a function of $R_n(J_n(i),J'_n(j))$,
this turns out
to be very easy, namely, since both chains start in the invariant distribution
\begin{eqnarray}\label{vero.1}
&&\EE_{\pi_n}\EE'_{\pi_n} \1_{nR_n(J_n(i),J'_n(j))=m}\nonumber\\
&&\qquad=\sum_{x,y\in\SS^n}\PP_{\pi_n}\bigl(J_n(i)=x\bigr)\PP'_{\pi_n}\bigl(J'_n(i)=y\bigr)
\1_{nR_n(x,y)=m}
\\
&&\qquad= 2^{-n} \sum_{x\in\SS^n} \1_{nR_n(x,1)=m}
=2^{-n}\pmatrix{n\cr(n-m)/2}.\nonumber
\end{eqnarray}
Putting all things together, we arrive at the bound
\begin{eqnarray}\label{jiri.25}\nonumber
&&k_n(t)^2\bigl|\E G(V^0,v,2\th)-(\E G(V^0,v,\th)
)^2\bigr|
\nonumber\hspace*{-35pt}\\
&&\quad\leq
\sum_{m= 0}^n\! 2^{-n}\!\pmatrix{n\cr(n-m)/2}\!\!
\biggl(\frac mn\biggr)^p t^2ne^{n\g^2/\b^2} 2C
\exp\biggl(-\frac{n\g^2}{\b^2(1+(m/n)^p)}\biggr)\hspace*{-35pt}\\
&&\qquad{}+
\sum_{m=0}^n\! 2^{-n}\!\pmatrix{n\cr(n-m)/2}\!\! \biggl(\frac mn\biggr)^p
t^2ne^{n\g^2/\b^2} 2C
\exp\biggl(-\frac{n\g^2}{\b^2}\biggr),\nonumber\hspace*{-35pt}
\end{eqnarray}
where we did use that $k_n(t)\th_n \approx t \sqrt n e^{n\g^2/\b^2}$.
Clearly the second term is smaller than the first, so we only need to
worry about the latter. But this term is exactly the term (3.28)
in~\cite{BBC08}, where it is shown that this is smaller than
\begin{equation}\label{jiri.30}
C't^2 n^{1-p/2},
\end{equation}
provided $\g<\zeta(p)$. This provides the assertion of our
Lemma~\ref{jiri.11} and concludes its proof.
\end{pf}

\begin{remark*} The estimate on the second moment we get here allows to
get almost sure convergence only if $p>4$. It is not quite clear
whether this is natural. We were tempted to estimate higher moments to
get improved estimates on the convergence speed. However, any straightforward
application of the comparison methods used here does produce the same order
for all higher moments. We have not been able to think of a tractable way
to improve this result.
\end{remark*}

\subsection{\texorpdfstring{Verification of the second part of Condition \protect\ref{coa2.1}}
{Verification of the second part of Condition (A2-1)}}
For $u,u'>0$ define
\begin{eqnarray}\label{sigma.1}
\tilde\eta_n^t(u)&=& \frac{1}{n}k_n(t)\sum_{x\in\VV_n}\pi
_n(x)(Q^{u}_n(x))^2,
\\
\eta_n^t(u,u')&=&k_n(t)\sum_{x\in\VV_n}\sum_{x'\in\VV_n}\mu
_n(x,x')Q^{u}_n(x)Q^{u'}_n(x'),
\label{sigma.2}
\end{eqnarray}
where $\mu_n$ is the uniform distribution on pairs of vertices $(x,x')$
that are at distance~2 apart,
\begin{equation}\label{sigma.3}
\mu_n(x,x')=
\cases{ 2^{-n}\displaystyle\frac{2}{n(n-1)}, &\quad$\mbox{if } \dist
(x,x')=2,$\vspace*{2pt}\cr
0,&\quad$\mbox{else}.$}
\end{equation}
Equation (\ref{3.prop1.01}) will be verified if we can show that for all $t>0$ and
all $u,u'>0$, both $\tilde\eta_n^t(u)$ and $\eta_n^t(u,u')$
tend to zero, almost surely, respectively, in probability, as
$n\uparrow\infty$.

As before we will do this by first passing to the Laplace transform of
$\eta_n^t(u,u')$. For $v,v'>0$, define
\begin{eqnarray}
\hat\eta_n^t(v,v')&=&\int_{0}^{\infty} du\int_{0}^{\infty
}du'e^{-(uv+u'v')}\eta_n^t(u,u'),
\nonumber
\\[-8pt]
\\[-8pt]
\nonumber
\hat\eta^t(v,v')&=&\int_{0}^{\infty} du\int_{0}^{\infty
}du'e^{-(uv+u'v')}\eta^t(u,u').
\label{sigma.5}
\end{eqnarray}
The reason for considering the two point function $\eta_n^t(u,u')$ is
that, integrating by parts as in (\ref{pspin.12}), $\hat\eta_n^t(v,v')$ takes
the convenient form
\begin{equation}
\qquad\hat\eta_n^t(v,v')
=k_n(t)\sum_{x\in\VV_n}\sum_{x'\in\VV_n}\mu_n(x,x')
\frac{1-\EE_{x}(e^{-vZ_n})}{v}\frac{1-\EE'_{x'}
(e^{-v'Z'_n})}{v'},
\label{sigma.6}
\end{equation}
where $\EE_{x}$ (resp., $\EE'_{x'}$) denotes the expectation with
respect to the law $\PP_x$ of the chain $X_n$ started in $x$
(resp., the law $\PP'_{x'}$ of an independent copy $X'_n$ started in $x'$).

\begin{lemma}\label{sigma.lem.1}
Under the assumptions, and with the notation of Lemma~\ref{pspin.11},
for any $v,v'>0$,
\begin{equation}\label{sigma.7}
\lim_{n\uparrow\infty} \E\hat\eta_n^t(v,v')
=0.
\end{equation}
\end{lemma}

\begin{pf} The key idea of the proof is that the first $\bar\th_n=
2n\ln n $ terms
in the sums $Z_n$ are irrelevant. With this in mind, we define
$W_n\equiv\sum_{j=\bar\th_n}^{\th_n-1} c_n^{-1}\times {\lambda
}_n^{-1}(J_n(j))e_{n,j}$.

Note that
\begin{eqnarray}\label{sigma.8}
vv'\E\hat\eta_n^t(v,v')&=&k_n(t)\E[1-\EE_{\pi_n}
(e^{-vZ_n})]
+k_n(t)\E[1-\EE'_{\pi_n}(e^{-v'Z'_n})\nonumber]\\
&&{}-
k_n(t)\sum_{x\in\VV_n}\sum_{x'\in\VV_n}\mu_n(x,x')
\E\bigl[1-\EE_{x}\EE'_{x'}\bigl(e^{-(vZ_n+v'Z'_n)}\bigr)\bigr]
\nonumber
\\[-8pt]
\\[-8pt]
\nonumber
&\leq&
k_n(t)\E[1-\EE_{\pi_n}(e^{-vZ_n})]
+k_n(t)\E[1-\EE'_{\pi_n}(e^{-v'Z'_n})]\\
\nonumber
&&{}-
k_n(t)\sum_{x\in\VV_n}\sum_{x'\in\VV_n}\mu_n(x,x')
\E\bigl[1-\EE_{x}\EE'_{x'}\bigl(e^{-(vW_n+v'W'_n)}\bigr)
\bigr].\nonumber
\end{eqnarray}
Adding and subtracting the term $\E\EE_{x}( e^{-vW_n})\E
\EE
'_{x'}(e^{-v'W'_n})$
to the term\break $\E\EE_{x}\EE'_{x'}
(e^{-(vW_n+v'W'_n)})$, the right-hand side of (\ref{sigma.8}) is
equal to
\begin{eqnarray}\label{sigma.8'}
&&\quad k_n(t)\E[1-\EE_{\pi_n}(e^{-vZ_n})]
+k_n(t)\E[1-\EE'_{\pi_n}(e^{-v'Z'_n})]\\
&&{}-
k_n(t)\sum_{x\in\VV_n}\sum_{x'\in\VV_n}\mu_n(x,x')
[1-\E\EE_{x}( e^{-vW_n})\E\EE'_{x'}
(e^{-v'W'_n}
)]
\nonumber\\
&&{}+
k_n(t)\sum_{x\in\VV_n}\sum_{x'\in\VV_n}\mu_n(x,x')
\bigl(\E\EE_{x}\EE'_{x'}\bigl(e^{-(vW_n+v'W'_n)}\bigr)\nonumber\\[-2pt]
&&{}-\E\EE_{x}( e^{-vW_n})\E\EE'_{x'}
(e^{-v'W'_n})
\bigr).\nonumber
\end{eqnarray}
After a little reorganisation, (\ref{sigma.8'}) is in turn equal to
\begin{eqnarray}\label{sigma.8''}
&&k_n(t)\E[\EE_{\pi_n}(e^{-vW_n}-e^{-vZ_n}+e^{-v'W_n}-e^{v'Z_n})]\hspace*{-15pt}\\[-2pt]
&&{}+vv'k_n(t)\sum_{x\in\VV_n}\sum_{x'\in\VV_n}\mu_n(x,x')
\E\frac{1-\EE_{x}(e^{-vW_n})}{v}\E\frac{1-\EE'_{x'}(e^{-v'W'_n})}{v'}\nonumber\hspace*{-15pt}\\[-2pt]
&&{}+k_n(t)\sum_{x\in\VV_n}\sum_{x'\in\VV_n}\mu_n(x,x')
\EE_x\EE_{x'}\bigl(\E\bigl(e^{-(vW_n+v'W'_n)}\bigr)\hspace*{-15pt}\nonumber\\[-2pt]
&&{}-\E( e^{-vW_n})\E(e^{-v'W'_n})\bigr).\nonumber\hspace*{-15pt}
\end{eqnarray}
Now one deduces readily from Lemma~\ref{pspin.11} that
\begin{equation}\label{new.1}
k_n(t)\E[\EE_{\pi_n}(e^{-vW_n}-e^{-vZ_n})
]\sim
K_ptv^{\g/\b^2}\bar\th_n/\th_n =O\biggl(\frac{\ln n}n\biggr)
\end{equation}
and tends to zero as $n\uparrow\infty$.
Also by Lemma~\ref{pspin.11},
\begin{eqnarray}\label{new.2}
&&k_n(t)\sum_{x\in\VV_n}\sum_{x'\in\VV_n}\mu_n(x,x')
\E\frac{1-\EE_{x}(e^{-vW_n})}{v}\E\frac{1-\EE
'_{x'}
(e^{-v'W'_n})}{v'}
\nonumber
\\[-9pt]
\\[-9pt]
\nonumber
&&\qquad =O\bigl(1/k_n(t)\bigr)
\end{eqnarray}
and tends to zero even much faster. The last term in (\ref{sigma.8''})
will be controlled
by the Gaussian comparison method similar to the proof of Lemma~\ref{jiri.11}.
Indeed, using the same comparison and interpolation process as in the
proof of
that lemma, we see that for given trajectories $J_n, J_n'$,
\begin{eqnarray}\label{new.4}
&&\E\bigl(e^{-(vW_n+v'W'_n)}\bigr)
-\E( e^{-vW_n})\E(e^{-v'W'_n})
\nonumber
\\[-9pt]
\\[-9pt]
\nonumber
&&\qquad= \int_0^1 \mathop{\sum_{\bar\th_n\leq i<\th_n}}_{
{\th_n+\bar\th_n\leq j<2\th_n}} \L_{ij}^0 \E\frac{\del^2
G(V^h,v,2\th_n)}
{\del v_i\,\del v_j}\,dh.
\end{eqnarray}
To control the right-hand side we will exploit the fact that
after $\OO(n\log n)$ steps, such trajectories are at maximal distance
apart with probability close to one.
Recalling (\ref{graph.dist}), define the distance chain, $D_n$,
on $\{0,1,\ldots,n\}$ through
\begin{equation}\label{sigma.9}
D_n(i)=\dist(J_n(i),J'_n(i)),\qquad i\geq1.
\end{equation}

\begin{lemma}\label{sigma.lem.2} Set $\bar\th_n= 2n\log n$ and $\rho
(n)=\sqrt{K\frac{\log n}{n}}$.
Then, for $K$ sufficiently large,
\begin{equation}\label{sigma.10}
P\biggl(\forall_{\bar\th_n\leq i\leq\th_n}D_n(i)>\frac
{n}{2}\bigl(1-\rho(n)\bigr)
\vert  D_n(0)=2\biggr)\geq1-n^{-8}.\vadjust{\goodbreak}
\end{equation}
Moreover, for any fixed $x,y\in\VV_n$,
\begin{equation}\label{sigma.10'}
P_x\biggl(\exists_{\bar\th_n\leq i\leq\th_n}\dist(J_n(i),y)<\frac
{n}{2}\bigl(1-\rho(n)\bigr)\biggr)\leq\frac{1}{n^4}.
\end{equation}
\end{lemma}

\begin{pf}
Observe on the one hand that, denoting by $\DD_n$, the transition
matrix of
the distance chain $\DD_n$, one has $\DD_n=(\QQ_n)^2$, where $\QQ
_n$ is
the transition matrix
of the Ehrenfest chain on state space $\{0,\ldots,n\}$, namely, the
chain with
transition probabilities $q_n(i,i+1)=\frac{i}{n}$ and $q_n(i,
i-1)=1-\frac{i}{n}$.
On the other hand, it is sufficient in order to prove (\ref{sigma.10'})
to prove it for $y=\mathbf{1}\equiv(1,\ldots,1)$,
and again, the projection chain $\Pi_n(i)\equiv\dist(J_n(i),\mathbf{1})$,
$i\geq1$, is nothing but the Ehrenfest chain on $\{0,\ldots,n\}$.
Both equations (\ref{sigma.10}) and (\ref{sigma.10'}) then follow
from well-known
estimates for the Ehrenfest chain; specifically, see~\cite{Kem}, page
25, equation below~(4.18).
\end{pf}

Let $\AA_n\subset\FF^{J}\times\FF^{J'}$ be the event
$
\AA_n\equiv\{\forall_{\bar\th_n\leq i\leq\th
_n}D_n(i)>\frac
{n}{2}(1-\rho(n))\}
$.
Notice first that on $\AA_n$, by the estimates in Lemma~\ref{jiri.19},
\begin{eqnarray}\label{new.5}
\int_0^1 \mathop{\sum_{\bar\th_n\leq i<\th_n}}_{
{\th_n+\bar\th_n\leq j<2\th_n}} \L_{ij}^0
\E\frac{\del^2 G(V^h,v,2\th_n)}
{\del v_i\,\del v_j}\,dh &\leq&2C \th_n^2 \rho(n) \exp\bigl(-\g^2 n/\b
^2\bigr)
\nonumber
\\[-8pt]
\\[-8pt]
\nonumber
&=&O(k_n(t)^{-2}).
\end{eqnarray}
On the other hand, on $\AA_n^c$, we still have the bound
\begin{eqnarray}\label{new.6}
&&\int_0^1 \mathop{\sum_{\bar\th_n\leq i<\th_n}}_{{\th_n+\bar\th
_n\leq j<2\th_n}} \L_{ij}^0
\E\frac{\del^2 G(V^h,v,2\th_n)}
{\del v_i\,\del v_j}\,dh \nonumber\\
&&\qquad\leq2C \th_n^2 \exp\bigl(-\g^2 n/2\b^2\bigr)
\\
&&\qquad=O\bigl(\th_n /k_n(t)\bigr).\nonumber
\end{eqnarray}
Putting all estimates together we arrive at the assertion of the lemma.
\end{pf}

To prove convergence in probability, respectively, almost surely, we
just need to
use the same concentration estimate as in Lemma~\ref{jiri.11} for the term
$k_n(t) \EE_{\pi_n}(e^{-vW_n}-e^{-vZ_n})$. Finally, the term
$\tilde\eta_n^t(u)$ from (\ref{sigma.1}) can be controlled in exactly
the same way. This establishes
Condition~\ref{coa2.1}.

\subsection{\texorpdfstring{Verification of Condition \protect\ref{coa3.1}}{Verification of Condition (A3-1)}}

To show that Condition~\ref{coa3.1} holds, we again first prove that
the average
of the
right-hand side vanishes as ${\varepsilon}\downarrow0$, and then we prove
a concentration result.

\begin{lemma} \label{a3.1}
Under the assumptions of the theorem, there is a constant $K<\infty$, such
that
\begin{equation}\label{a3.2}
\limsup_{n\uparrow\infty}a_n c_n^{-1} \E\EE_{\pi_n}
{\lambda}_n^{-1}(J_n(0)) e_0 \1_{ {\lambda}_n^{-1}(J_n(0)) e_0\leq
{\varepsilon}c_n}
\leq K{\varepsilon}^{1-\a}.
\end{equation}
\end{lemma}

\begin{pf}
The proof is through explicit estimates. We must control the integral
\begin{eqnarray}\label{a3.3}
&&\int_0^\infty xe^{-x}\,dx\int_{-\infty}^\infty
e^{-{z^2}/2} \1_{xe^{\b\sqrt n z}\leq{\varepsilon}c_n}
e^{\b\sqrt n z}\,dz\nonumber\\
&&\qquad=
\int_0^\infty xe^{-x}\,dx\biggl[
\int_{-\infty}^{{(\ln c_n+\ln({\varepsilon}/x))}/{(\b\sqrt n})}
e^{-{z^2}/2+\b\sqrt n z}\,dz\biggr]\\
&&\qquad=
\int_0^\infty xe^{-x}\,dx\biggl[e^{\b^2n/2}
\int_{-\infty}^{({(\ln c_n+\ln({\varepsilon}/x))}/{(\b\sqrt n)})-\b
\sqrt n}
e^{-{z^2}/2}\,dz\biggr].\nonumber
\end{eqnarray}
Now for our choice $c_n=\exp(\g n)$, the upper integration limit in
the $z$-integral is
\begin{eqnarray}\label{a3.4}
&&\frac{\ln c_n+\ln({\varepsilon}/x)}{\b\sqrt n}-\b\sqrt n
\nonumber
\\[-8pt]
\\[-8pt]
\nonumber
&&\qquad=\sqrt n\biggl(\frac\g\b-\b\biggr) +\frac{\ln{\varepsilon}-\ln
x}{\b\sqrt n}.
\end{eqnarray}
Thus, for any $\g<\b^2$, this tends to $-\infty$ uniformly for, say,
all $x\leq n^2$.
We therefore decompose the $x$-integral in the domain $x\leq n^2$ and
its complement, and use first that
\begin{eqnarray}\label{a3.5}
&&\int_{n^2}^\infty x e^{-x}\,dx\int
e^{-{z^2}/2} \1_{xe^{\b\sqrt n z}\leq{\varepsilon}c_n}
e^{\b\sqrt n z}\,dz
\nonumber
\\[-8pt]
\\[-8pt]
\nonumber
&&\qquad\leq{\varepsilon}n^2 c_n e^{-n^2},
\end{eqnarray}
which tends to zero, as $n\uparrow\infty$. For the remainder we use the
bound
%
%
\begin{equation}
\int_u^\infty e^{-z^2/2}\leq
\frac{1}{u}e^{-u^2/2}.
\label{ext-2.10}
\end{equation}
This yields
\begin{eqnarray}\label{a3.6}
&&e^{\b^2n/2}
\int_{-\infty}^{({(\ln c_n+\ln({\varepsilon}/x))}/{(\b\sqrt n)})-\b
\sqrt n}
e^{-{z^2}/2}\,dz
\nonumber\\
&& \qquad\leq
e^{\b^2 n/2} \frac{
\exp(-1/2(\sqrt n(\b-\g/\b) -
{(\ln{\varepsilon}-\ln x)}/{\b\sqrt n})^2)}
{(\b-\b^{-1}\g)\sqrt n-{(\ln{\varepsilon}-\ln x)}/{\b\sqrt
n}}
\nonumber
\\[-8pt]
\\[-8pt]
\nonumber
&&\qquad=\frac{\exp( -n({\g^2}/{2\b^2})+n\g) }
{\sqrt n(\b-\g/\b)+o(1)}\exp\bigl(-(\g/\b^2-1)\ln( {\varepsilon}/x)
+O(n^{-1/2})\bigr)\\
&&\qquad
= c_n a_n^{-1}
\frac{1}
{\b-\g/\b+o(1)}\exp\bigl(-(\g/\b^2-1)\ln( {\varepsilon}/x)
+O(n^{-1/2})\bigr).\nonumber
\end{eqnarray}
Hence
\begin{eqnarray}\label{a3.7}
&&\limsup_{n\uparrow\infty}a_nc_n^{-1}
\int_0^\infty xe^{-x}\,dx\int_{-\infty}^\infty
e^{-{z^2}/2} \1_{xe^{\b\sqrt n z}\leq{\varepsilon}c_n}
e^{\b\sqrt n z}\,dz
\nonumber
\\[-8pt]
\\[-8pt]
\nonumber
&&\qquad\leq
\frac1{\b-\g/\b}{\varepsilon}^{1-\a} \int_0^\infty x^{\a}e^{-x}
\,dx,
\end{eqnarray}
where $\a=\g/\b^2$. This yields the assertion of the lemma.
\end{pf}

To complete the proof, we need a concentration
estimate.
The first step is a simple Gaussian bound.

\begin{lemma} \label{anton.1}
Let $X,Y$ be centered normal Gaussian random variables with covariance
$\E X Y = c$.
Then, for any ${\varepsilon},{\varepsilon}'>0$,
\begin{eqnarray}\label{anton.3}
\qquad&&\frac{\E( e^{\b\sqrt n X}\1_{e^{\b\sqrt n X}\leq c_n {\varepsilon}}
e^{\b\sqrt n Y}\1_{e^{\b\sqrt n Y}\leq c_n{\varepsilon}'})
}{
\E( e^{\b\sqrt n X}\1_{e^{\b\sqrt n X}\leq c_n{\varepsilon}})
\E( e^{\b\sqrt n Y}\1_{e^{\b\sqrt n Y}\leq c_n{\varepsilon}'})}
\nonumber
\\[-8pt]
\\[-8pt]
\nonumber
&&\qquad\leq
\frac{\exp(|c| {(2\g^2 n +\g(\ln{\varepsilon}+\ln
{\varepsilon}'))}/{(2\b
^2(1+|c|))})}
{\sqrt{1-c^2}-(\g/{\b^2})\sqrt{({1-|c|)}/{(1+|c|)}}}
\bigl(1+O(1/n)\bigr).
\end{eqnarray}
\end{lemma}

\begin{pf}
The numerator on the left-hand side of (\ref{anton.3}) equals (we assume
$c\geq0$ below,
but the same estimate with $c$ replaced by $-c$ can be obtained for $c<0$)
\begin{eqnarray*}
\label{anton.5}
&&\frac1{2\pi}\int_{-\infty}^{{( \g n +\ln{\varepsilon})}/{(\b
\sqrt n)}}
\int_{-\infty}^{{ (\g n +\ln{\varepsilon}')}/{(\b\sqrt n)}}
\frac1{\sqrt{1-c^2}}e^{-{(z_1^2+z^2_2+2cz_1z_2)}/{(2(1-c^2))}}\nonumber\\
&&\hspace*{143pt}\qquad{}\times e^{
\sqrt n(z_1+z_2)}
\,dz_1\,dz_2\\
&&\qquad=
\frac1{2\pi\sqrt{1-c^2}}\int_{-\infty}^{{( \g n +\ln
{\varepsilon})}/{\b
\sqrt n}}
\int_{-\infty}^{{ (\g n +\ln{\varepsilon}')}/{(\b\sqrt n)}}
e^{\b\sqrt n(z_1+z_2)}e^{-{(z_1^2+z^2_2)}/2}\\
&&\qquad\quad{}\times e^{-{(c(z_1-z_2)^2-c(1-c)(z_1^2+z^2_2))}/{2(1-c^2)}}\,dz_1\,dz_2
\\
&& \qquad
\leq
\frac1{2\pi\sqrt{1-c^2}}\int_{-\infty}^{{( \g n +\ln
{\varepsilon})}/{(\b
\sqrt n)}}
\int_{-\infty}^{{ (\g n +\ln{\varepsilon}')}/{(\b\sqrt n)}}
e^{\b\sqrt n(z_1+z_2)}e^{-{(z_1^2+z^2_2)}/2}\\
&&\qquad\quad{}\times e^{+{(c(z_1^2+z^2_2))}/{(2(1+c))}}\,dz_1\,dz_2\\
&&
\qquad=\frac1{ 2\pi\sqrt{1-c^2}}
\biggl(\int_{-\infty}^{{( \g n +\ln{\varepsilon})}/{(\b\sqrt n)}}
e^{\b\sqrt n z}e^{(-{z^2}/{(2(1+c))})}\,dz\biggr)\\
&&\quad\qquad{}\times \biggl(\int_{-\infty}^{{( \g n +\ln{\varepsilon}')}/{(\b\sqrt n)}}
e^{\b\sqrt n z}e^{-({z^2}/{(2(1+c))})}\,dz\biggr).
\end{eqnarray*}
Using standard estimates on the asymptotics of one-dimensional
Gaussian integrals the claimed
result follows after some straightforward computations.
\end{pf}

We will now use Lemma~\ref{anton.1} to prove the desired concentration
estimate.

\begin{lemma}\label{anton.10}
With the notation above,
\begin{eqnarray}\label{anton.11}
&&\E\bigl(\EE_{\pi_n} {\lambda}_n^{-1}(J_n(0))e_0\1_{{\lambda}
_n^{-1}(J_n(0))e_0\leq
{\varepsilon}c_n}\bigr)^2\nonumber\\
&&\quad{}- \bigl(\E\EE_{\pi_n} {\lambda}_n^{-1}(J_n(0))e_0\1_{{\lambda}
_n^{-1}(J_n(0))e_0\leq
{\varepsilon}c_n}\bigr)^2
\\
&&\qquad\leq C n^{1-p/2} \bigl(\E\EE_{\pi_n} {\lambda
}_n^{-1}(J_n(0))e_0\1
_{{\lambda}
_n^{-1}(J_n(0))e_0\leq
{\varepsilon}c_n}\bigr)^2.\nonumber
\end{eqnarray}
\end{lemma}

\begin{pf}
Writing out everything explicitly, we have\vspace*{1pt}
\begin{eqnarray}\label{anton.12}
&&\E\bigl(\EE_{\pi_n} {\lambda}_n^{-1}(J_n(0))e_0\1_{{\lambda}
_n^{-1}(J_n(0))e_0\leq
{\varepsilon}c_n}\bigr)^2\nonumber\\
&&\quad{}- \bigl(\E\EE_{\pi_n} {\lambda}_n^{-1}(J_n(0))e_0\1_{{\lambda}
_n^{-1}(J_n(0))e_0\leq
{\varepsilon}c_n}\bigr)^2
\nonumber\\
&&\qquad=2^{-2n} \sum_{x,x'\in\S_n} \int dy_1\,dy_2e^{-y_1-y_2} y_1y_2\\
&&\qquad\quad{}\times\bigl(\E\bigl( e^{\b(H_n(x)+H_n(x'))} \1_{e^{\b H_n(x)}\leq c_n
{\varepsilon}/y_1} \1_{e^{\b H_n(x')}\leq c_n
{\varepsilon}/y_2}\bigr)\nonumber\\
&&\hspace*{1pt}\phantom{\times\bigl(}\qquad\quad{} -
\E\bigl( e^{\b H_n(x)} \1_{e^{\b H_n(x)}\leq c_n
{\varepsilon}/y_1}\bigr)\E\bigl( e^{\b H_n(x')} \1_{e^{\b
H_n(x')}\leq c_n
{\varepsilon}/y_2}\bigr)\bigr).\nonumber
\end{eqnarray}
Now the last terms depend only on the covariance of $H_n(x)$ and
$H_n(x')$, that is,
on $R_n(x,x')$. Using Lemma~\ref{anton.1}, we get, when $R_n(x,x')^p=c$,\vspace*{1pt}
\begin{eqnarray}\label{anton.13}
&&
\int dy_1\,dy_2e^{-y_1-y_2} y_1y_2
\nonumber\\
&&\quad{}\times\bigl(\E\bigl( e^{(\b H_n(x)+H_n(x'))} \1_{e^{\b
H_n(x)}\leq c_n
{\varepsilon}/y_1} \1_{e^{\b H_n({{\sigma}'})}\leq c_n
{\varepsilon}/y_2}\bigr)
\nonumber
\\[-8pt]
\\[-8pt]
\nonumber
&&\hspace*{1pt}\phantom{\times\bigl(}\quad{}-
\E\bigl( e^{\b H_n(x)} \1_{e^{\b H_n({\sigma})}\leq c_n
{\varepsilon}/y_1}\bigr)\E\bigl( e^{\b H_n({\sigma}')} \1_{e^{\b
H_n({\sigma
}')}\leq c_n
{\varepsilon}/y_2}\bigr)\bigr)\\
&&\qquad\leq
(e^{cn ({\g^2}/{(\b^2(1+c))})}-1)
\bigl(\E\EE_{\pi_n} e^{\b H_n({\sigma})}e_1\1_{e^{\b H_n({\sigma
})}e_1\leq{\varepsilon}
}
\bigr)^2 \bigl(1+O(c)\bigr).\nonumber
\end{eqnarray}
Thus we have to control
\begin{eqnarray}\label{anton.14}
&&2^{-2n} \sum_{m\in\{-1,-1+{2}/{n},\ldots,1-
{2}/{n},1
\}} \sum_{x,x'\in\S_n} \1_{R_n(x,x')=m}
\bigl(e^{ m^p n ({\g^2}/(\b^2(1+m^p)))}-1\bigr)
\nonumber
\\[-8pt]
\\[-8pt]
\nonumber
&&\qquad= \sum_{m\in\{-1,-1+{2}/{n},\ldots,1-
{2}/{n},1\}}
2^{-n}\pmatrix{n\cr n(m+1)/2}
\bigl(e^{ m^p n ({\g^2}/{(\b^2(1+m^p))})}-1\bigr).
\end{eqnarray}
The analysis of the last sum can be carried out in the same way as was
done in
\cite{BBC08} for a very similar sum. It yields that
\begin{equation}\label{anton.15}
\sum_{m=-1}^1 2^{-n}\pmatrix{n\cr n(m+1)/2}
\bigl(e^{ m^p n {\g^2}/(\b^2(1+m^p))}-1\bigr) = C n^{1-p/2}.
\end{equation}
\upqed\end{pf}

\subsection{Conclusion of the proof}
Consider first the case $p>4$.
Lemmata~\ref{pspin.11} and~\ref{jiri.11}, together with
Chebyshev's inequality and the
Borel--Cantelli lemma, establish that for each $v>0$,
\begin{equation}
\lim_{n\rightarrow\infty}\hat\nu_n^{t}(v)=\hat\nu
^{t}(v)=K_pv^{\g/\b
^2-1} ,\qquad\P\mbox{-a.s.}
\end{equation}
Together with the monotonicity of $\hat\nu^{t}_n(v)$ and the continuity
of the limiting function
$\hat\nu^{t}(v)$, this implies that there exists a subset $\O
_1\subset\O
$ of
the sample space $\O$ of the~$\t$s with the property that $\P(\O_1)=1$,
and such that,
on $\O_1$,
\begin{equation}
\lim_{n\rightarrow\infty}\hat\nu_n^{t}(v)=\hat\nu^{t}(v)
\qquad\forall v>0 .
\label{1}
\end{equation}
Finally, applying Feller's extended continuity theorem for Laplace
transforms of
(not necessarily bounded) positive measures (see~\cite{feller2},
Theorem 2a, Section XIII.1, page 433)
we conclude that, on $\O_1$,
\begin{equation}
\lim_{n\rightarrow\infty}\nu_n^{t}(u,\infty)=\nu^{t}(u,\infty)=K_p
u^{-\g/\b^2}
\qquad\forall u> 0 .
\label{2}
\end{equation}

In the cases $p=3,4$, where our estimates give only convergence in
probability, we obtain convergence of $\nu_n^t(u,\infty)$ in probability,
for example, by using the characterization of convergence of
probability in terms of
almost sure convergence of sub-sequences; see, for example,~\cite{RW1},
Section II. 19.
This allows us to reduce the proof in this case to that of the case of
almost sure convergence.

Thus we have established Conditions~\ref{coa1.1},~\ref{coa2.1} and~\ref{coa3.1} under the
stated conditions on the parameters $\g,\b,p$, and Theorem~\ref{p:main}
follows from Theorem~\ref{main.1}.

\subsection{Consequences for correlation functions}

We now turn to the proof of Theorem~\ref{t:aging}.

\begin{pf}
The proof of this theorem relies on the following simple estimate.
Let us denote by $\RR_n$ the range of the coarse grained and rescaled
clock process $ S^b_n$. The argument of~\cite{BBC08} in the proof of
Theorem 1.2 that
the event $A^{\varepsilon}_n(s,t)\cap\{\RR_n\cap(s,t)\neq
\varnothing\}$ has vanishing\vadjust{\goodbreak}
probability carries over unaltered. However,
while in their case, $A^{\varepsilon}_n(s,t)\supset\{\RR_n\cap
(s,t)=\varnothing
\}$,
was obvious, due to the fact that the coarse graining was done on a scale
$o(n)$; this is not immediately clear in our case, where the number of steps
within a block is of order $n^2$. What we have to show is that
if the process spends the whole time from $s$ to $t$ within one bloc, then
almost all of this time is spent, without interruption, within a small ball
of radius ${\varepsilon}n$.

To show that this holds, we will need to establish two facts.

\begin{fact}\label{fa1}
The first fact concerns the random environment.
We will show that, if
a trajectory within a block of length $\th_n\sim n^2$ hits a point
where the
random variables $H_n$ are ``big,'' that is, of order $an$, then with
overwhelming probability, all other sites with ``big'' $H_n$s this
piece of path meets are
within a distance ${\varepsilon}n$ from this point. In other words,
within one block,
the path will never hit two distinct clusters of large values of the
random field.
\end{fact}

\begin{fact}\label{fa2}
The second fact concerns the properties of the random walk $J_n$. We
will show
that the random walk that hits such a cluster of large values will
spend there,
at most, a time of order ${\varepsilon}n$, and it will not leave
that cluster and return to it later within $\th_n$ steps.
\end{fact}

These two properties imply our claim.

The proof of the first fact relies on the following elementary estimate
for correlated Gaussian variables.
Note that the following bound is not optimal but good enough for our
purposes.

\begin{lemma}\label{simple.1}
Let $X,Y$ be standard Gaussian variables with covariance $\cov(X,Y)= 1-c$,
$0<c< 1/4$. Then for $a>0$,
\begin{eqnarray}\label{simple.2}
\P\bigl(X>a, Y>a(1-c/4)\bigr)&\leq&
\frac1{a 2\pi\sqrt c}\exp\biggl(-\frac{a^2}2\biggl(1+\frac
c{32}
\biggr)\biggr)
\nonumber
\\[-8pt]
\\[-8pt]
\nonumber
&&{}+\frac1 {\sqrt{2\pi} a }\exp\biggl(-\frac{a^2}2
(1+c
)\biggr).
\end{eqnarray}
\end{lemma}

\begin{pf} Note that the variables $X,Y$ have the joint density
\begin{equation}\label{simple.3}
\frac1{2\pi\sqrt{2c-c^2}} \exp\biggl(-\frac{x^2}2 -\frac
{(y-(1-c)x)^2}{4c-2c^2}\biggr).
\end{equation}
Next,
\begin{eqnarray}\label{simple.4}
&&\P\bigl(X>a, Y>a(1-c/2)\bigr)
\nonumber
\\[-8pt]
\\[-8pt]
\nonumber
&&\qquad\leq
\P\bigl(X>a, |Y-(1-c)X|>ac/4\bigr)+
\P\biggl(X>a \frac{1-c/2}{1-c}\biggr).
\end{eqnarray}
The result is now a trivial application of the standard tail estimates for
Gaussian integrals.\vadjust{\goodbreak}
\end{pf}

This lemma has the following corollary, which is a precise statement of
Fact~\ref{fa1}.

\begin{corollary}\label{gauss.1}
Let $H_n({\sigma})$ be the Gaussian process defined in (\ref
{pspin.1}). Let
$\MM_n\subset\S_n$ be arbitrary. Then, for ${\varepsilon}>0$ and
all $n$ large enough,
\begin{eqnarray}\label{gauss.2}
&&\P\bigl(\exists_{x,x'\in\MM_n}: R_n(x,x')<1-{\varepsilon}\mbox{ and}\nonumber\\
&&\quad
H_n(x)\geq a n \land H_n(x')\geq a n (1-p{\varepsilon}/4)\bigr)
\\
&&\qquad\leq|\MM_n|^2 e^{-na^2/2 }e^{-na^2 p{\varepsilon}/64}.\nonumber
\end{eqnarray}
\end{corollary}

A precise version of the second fact is the following lemma.

\begin{lemma}\label{srw.101}
Define the events
\begin{equation}\label{srw.100}
\WW_{{\varepsilon}}(k)\equiv
\exists_{\{\th_n k\leq i<j-{\varepsilon}n\leq\th_n(k+1)\}}
\{R_n(J_n(i),J_n(j))\geq1-{\varepsilon}\}.
\end{equation}
Then, for any ${\varepsilon}<1/4$, there exists a constant $C<\infty
$, such that
for all $n$ large enough, there exists $c_{\varepsilon}>0$, such that
\begin{equation}\label{srw.102}
P_{\pi_n}
(\WW_{\varepsilon}(k))
\leq Ce^{-c_{\varepsilon}n}.
\end{equation}
\end{lemma}

\begin{pf}
We clearly have to show only that an estimate
of the form (\ref{srw.102}) holds, for any $j\geq{\varepsilon}n$ for
the probability
$P_{\pi_n}( R_n(J_n(0),J_n(j))\geq1-{\varepsilon})$. We may also assume
that $J_n(0)=\mathbf{1}\equiv(1,\ldots,1)$. Observing that
$R_n(x,x')= 1-2\dist(x,x')$ [see (\ref{graph.dist})], we have
$
P( R_n((\mathbf{1},J_n(j)))\geq 1-{\varepsilon})
= P( \dist((\mathbf{1},\break J_n(j)))< {\varepsilon}/2)
$.
Now we saw in the proof of Lemma~\ref{sigma.lem.2} that the chain
$\Pi_n(j)\equiv\dist(\mathbf{1},J_n(j))$, $j\geq1$, is the Ehrenfest
chain on $\{0,\ldots,n\}$,
and again the desired exponential estimate follows from well-known
estimates for the latter chain; see, for example,~\cite{Kem}.
\end{pf}

We now continue the proof of Theorem~\ref{t:aging}.
As remarked above,
\begin{eqnarray}\label{notso.1}
\PP_{\pi_n}(A_n^{\varepsilon}(s,t))
&=&\PP_{\pi_n}\bigl(A_n^{\varepsilon}(s,t)\cap\bigl\{\RR_n\cap
(s,t)=\varnothing\bigr\}
\bigr)
\nonumber
\\[-8pt]
\\[-8pt]
\nonumber
&&{}+\PP_{\pi_n}\bigl(A_n^{\varepsilon}(s,t)\cap\bigl\{\RR_n\cap
(s,t)\neq
\varnothing\bigr\}
\bigr),
\end{eqnarray}
where the second term tends to zero.
Next we observe that
\begin{eqnarray}\label{simple.10}
&&\PP_{\pi_n}\bigl(A_n^{\varepsilon}(s,t)\cap\bigl\{\RR_n\cap
(s,t)=\varnothing\bigr\}
\bigr)
\nonumber
\\[-8pt]
\\[-8pt]
\nonumber
&&\quad=\PP_{\pi_n}\bigl(\RR_n\cap(s,t)=\varnothing\bigr)
-\PP_{\pi_n}\bigl((A_n^{\varepsilon}(s,t))^c\cap\bigl\{\RR_n\cap
(s,t)=\varnothing\bigr\}\bigr).
\end{eqnarray}
Here the first term is what we want.
The event in the second term occurs only if the block-variable, that
ensures that the
event $\RR_n\cap(s,t) =\varnothing$ occurs, contains a very long
block or
two sub-blocks contributing to
its internal ``clock-time.''
Corollary~\ref{gauss.1} and Lemma~\ref{srw.101} will be used to prove
that this tends to zero.
To do so, it is convenient to first show that the jump over $(s,t)$ is
realized before $k_n(N)$ steps, with high probability.\vadjust{\goodbreak}

For any $N<\infty$, we have
\begin{eqnarray}\label{notso.2}
&&\PP_{\pi_n}\bigl((A_n^{\varepsilon}(s,t))^c\cap\bigl\{\RR
_n\cap
(s,t)=\varnothing
\bigr\}
\bigr)
\nonumber\\
&&\qquad=\sum_{k=0}^{k_n(N)-1}
\PP_{\pi_n}\bigl(((A_n^{\varepsilon}(s,t))^c)\cap\bigl\{
(s,t)\subset
\bigl( S^b_n(k),S^b_n(k+1)\bigr)\bigr\}\bigr)\\
&&\qquad\quad{}+
\sum_{k=k_n(N)}^\infty
\PP_{\pi_n}\bigl(((A_n^{\varepsilon}(s,t))^c)\cap\bigl
\{
(s,t)\subset
\bigl(S^b_n(k),S^b_n(k+1)\bigr)\bigr\}\bigr).\nonumber
\end{eqnarray}
The second term is bounded by
\begin{eqnarray}\label{notso.3}
&&\sum_{k=k_n(N)}^\infty
\PP_{\pi_n}\bigl(((A_n^{\varepsilon}(s,t))^c)\cap\bigl
\{
(s,t)\subset
\bigl( S^b_n(k), S^b_n(k+1)\bigr)\bigr\}\bigr)
\nonumber
\\[-8pt]
\\[-8pt]
\nonumber
&&\qquad\leq\PP_{\pi_n} \bigl( S^b_n(N)\leq s\bigr)\rightarrow
\PP\bigl(V_{\gamma/\beta^2}(N)\leq s\bigr),
\end{eqnarray}
where convergence is almost sure (respectively, in probability, if $p=3$ or $p=4$)
with respect to the environment, due to the already established
convergence of
$S_n^b$.
The last probability can be made as small as desired by choosing $N$
sufficiently large. It remains to deal with the first sum on the
right-hand side of (\ref{notso.2}).

For a given trajectory $J_n$, define the event, $\GG_\rho(k)\subset
\FF
^\t$,
that in
block number~$k$ (of size
$\th_n$) two points contribute
significantly to the clock that have overlap smaller then $1-\rho$.
More precisely,
\begin{eqnarray}\label{simple.12}
\GG_\rho(k)&\equiv&\mathop{\bigcup_{k\th_n\leq i<j<(k+1)\th_n}}_{
R_n(J_n(i),J_n(j))\leq1-\rho}\biggl\{ {\lambda}^{-1}_n(J_n(i))e_{n,i}
\geq\frac{c_n}{\th_n}(t-s) \biggr\}
\nonumber
\\[-8pt]
\\[-8pt]
\nonumber
&&{}\hspace*{80pt}\cap\biggl\{ {\lambda
}^{-1}_n(J_n(j))e_{n,j}
\geq\frac{c_n}{\th_n}n^{-1}\biggr\}.
\end{eqnarray}
Note that Corollary~\ref{gauss.1} implies that the probability of this event,
with respect to the law $\P$, is bounded nicely and uniformly in the
variables $J$.
Namely,
\begin{equation}\label{notso.6}
\E\PP_{\pi_n}(
\GG_\rho(k))\leq a_n^{-1} e^{-{\delta}n},
\end{equation}
for some ${\delta}>0$ depending on the choice of $\rho$. The simplest
way to
see this
is to use that the probability that one of the $e_{n,i}$ is larger than $n^2$
is smaller than $\exp(-n^2)$, and then use the bound from Corollary
\ref{gauss.1}.

On the other hand, on the event $\GG_\rho(k)^c$,
$ (A_n^{\varepsilon}(s,t))^c\cap\{(s,t)\subset
(S^b_n(k),\break S^b_n(k+1))\}$ can only happen if
the following are true:
first, there still must exist some $i$ such that
$ {\lambda}^{-1}_n(J_n(i))e_{n,i}
\geq c_n(t-s)\th_n^{-1}$, and second, the random walk must realize the event
considered in Lemma~\ref{srw.101}.

By these considerations, we have the bound
\begin{eqnarray}\label{simple.15}\qquad
&&\E\Biggl(\sum_{k=0}^{k_n(N)}\PP_{\pi_n}
\bigl((A_n^{\varepsilon}(s,t))^c\cap\bigl\{(s,t)\subset
\bigl(S^b_n(k),S^b_n(k+1)\bigr)\bigr\}\bigr)\Biggr)
\nonumber\\
&&\qquad\leq\sum_{k=0}^{k_n(N)} \E\bigl(\PP_{\pi_n}(
\GG_\rho(k))+\PP_{\pi_n}\bigl(\bigl\{\exists_{k\th_n\leq
i<(k+1)\th_n}
{\lambda}_n^{-1}(J_n(i))e_{n_i}>c_n \th_n^{-2}\bigr\}\\
&&\hspace*{258pt}\qquad{} \cap\WW
_{{\varepsilon}
}(k)\bigr)
\bigr).\nonumber
\end{eqnarray}

Next, we use Lemma~\ref{srw.101} and similar reasoning as before to
see that
\begin{eqnarray}\label{notso.7}
&&\E\PP_{\pi_n} \bigl(\bigl\{\exists_{k\th_n\leq i<(k+1)\th_n}
{\lambda}_n^{-1}(J_n(i))e_{n_i}>c_n \th_n^{-2}\bigr\} \cap\WW
_{{\varepsilon}
}(k)\bigr)
\nonumber\\
&&\qquad=
\P\bigl(\exists_{k\th_n\leq i<(k+1)\th_n}
{\lambda}_n^{-1}(J_n(i))e_{n_i}>c_n \th_n^{-2}\bigr)P_{\pi_n}( \WW
_{{\varepsilon}
}(k))
\nonumber
\\[-8pt]
\\[-8pt]
\nonumber
&&\qquad
\leq\th_n^2 \P\bigl(e^{\b H_n(x)}>c_n n^{-4}\bigr)
Ce^{-n{\varepsilon
}c_{\varepsilon}
} +
\th_n e^{-n^2}\\
&&\qquad\leq\th_n^2 a_n^{-1} n^{\g\sqrt n \b^{-2}}
e^{-n{\varepsilon}c_{\varepsilon}
} +
\th_n e^{-n^2}.\nonumber
\end{eqnarray}
Combining all this,
we see that
\begin{equation}\label{notsosimple.15}
\E\Biggl(\sum_{k=0}^{k_n(N)}\PP_{\pi_n}
\bigl((A_n^{\varepsilon}(s,t))^c\cap\bigl\{(s,t)\subset
\bigl(S^b_n(k),S^b_n(k+1)\bigr)\bigr\}\bigr)\Biggr)
\leq C Ne ^{-{\delta}n},\hspace*{-35pt}
\end{equation}
for some positive ${\delta}$, whatever the choice of ${\varepsilon}$.
But this estimate implies that the term (\ref{simple.10}) converges to zero
$\P$-almost surely, for any choice of $N$.
Hence the result is obvious from the $J_1$ convergence of $S^b_n$.
\end{pf}

%

\printaddresses

\end{document}